\documentclass[a4paper, 11pt]{article}

\usepackage{standalone, lmodern, xcolor, enumitem}
\usepackage{amsmath, amssymb, amsthm, physics, bbm, mathtools, stmaryrd, leftidx}
\usepackage{thmtools}
\usepackage{hyperref}

\usepackage{pgfplots}
\pgfplotsset{compat=1.16}
\usetikzlibrary{patterns}

\usepackage[
    hyperref,
    maxalphanames=5,
    maxbibnames=10,
    style=alphabetic,
    date=year,
    backend=biber,
]{biblatex}
\usepackage[new]{old-arrows}

\theoremstyle{plain}
\newtheorem{theorem}{Theorem}[section]
\newtheorem{proposition}[theorem]{Proposition}
\newtheorem{lemma}[theorem]{Lemma}
\newtheorem{corollary}[theorem]{Corollary}

\numberwithin{equation}{section}
\theoremstyle{definition}
\newtheorem{definition}[theorem]{Definition}
\newtheorem{example}[theorem]{Example}
\theoremstyle{remark}
\newtheorem{remark}[theorem]{Remark}

\numberwithin{figure}{section}

\DeclareEmphSequence{\bfseries}

\setlist[enumerate]{topsep = 1ex, leftmargin=.7cm, itemsep= -2pt}
\allowdisplaybreaks

\definecolor{color1}{HTML}{1b9e77}
\definecolor{color2}{HTML}{F1A340}
\definecolor{color3}{HTML}{5e3c99}
\hypersetup{
  colorlinks = true, %Colors links instead of ugly boxes
  urlcolor = color3, %Color for external hyperlinks
  linkcolor = color1, %Color of internal links
  citecolor = color3 %Color of citations
}

\newcommand{\R}{\mathbb{R}}
\newcommand{\Rp}{\mathbb{R}_{+}}
\newcommand{\Z}{\mathbb{Z}}
\newcommand{\C}{\mathbb{C}}
\newcommand{\ii}{\mathrm{i}\,}
\newcommand{\Zpos}{\mathbb{Z}_{>0}}
\newcommand{\vol}{\mathrm{vol}}
\newcommand{\dz}{|\dd z|^2}
\newcommand{\dw}{|\dd w|^2}
\newcommand{\dvol}[1]{\dd V_{#1}}
\newcommand{\dboundary}[1]{\dd \ell_{#1}}
\newcommand{\evaluation}{\mathsf{ev}}

\newcommand{\boldone}{\mathbbm{1}}
\DeclareMathOperator{\id}{{\boldone}}
\newcommand{\blank}{\:\cdot\:}
\newcommand{\setsuchthat}[2]{\left\{ {#1} \:\middle|\: {#2} \right\}}
\newcommand{\setsuchthatinline}[2]{\{ {#1} \:|\: {#2} \}}
\newcommand{\restrict}[2]{{#1}\vert_{{#2}}}
\newcommand{\const}{\mathrm{(const.)}}

\newcommand{\detz}[1]{{\det}_\zeta \, \Delta_{#1}}
\newcommand{\detzp}[1]{{\det'_\zeta} \, \Delta_{#1}}
\newcommand{\schwarzian}[1]{\mathcal{S}[{#1}]}
\newcommand{\preschwarzian}[1]{\mathcal{A}[{#1}]}
\newcommand{\genus}{\mathsf{g}}
\newcommand{\boundaries}{\mathsf{b}}
\newcommand{\moduli}[2]{\mathcal{M}_{{#1}, {#2}}}
\newcommand{\disk}{\mathbb{D}}
\newcommand{\annulus}{\mathbb{A}}
\DeclareMathOperator{\conformalanomaly}{S_L^0}
\DeclareMathOperator{\paformula}{PA}

\newcommand{\lpot}[1]{\mathcal{H}_{#1}}
\newcommand{\lpotx}[2]{\mathcal{H}^{#2}_{#1}}
\newcommand{\lenergy}[1]{\operatorname{I}_{\, {#1}}}

\newcommand{\charge}{\mathbf{c}}
\newcommand{\globaltriv}{Z}
\newcommand{\globaltrivcft}{Z}
\newcommand{\globaltrivzeta}{Z_\zeta^\charge}
\newcommand{\boundaryterm}[1]{\operatorname{B}_{#1}}

\newcommand{\SLE}{\mathrm{SLE}}
\newcommand{\sleloop}[1]{\mu^{\charge}_{#1}}
\newcommand{\sleloopx}[2]{\mu^{{\charge}, {#2}}_{#1}}
\newcommand{\loopnbhd}[2]{\mathcal{O}_{#1}({#2})}
\newcommand{\flatmetric}{\dd z \dd \bar z}
\newcommand{\interaction}{\operatorname{V}}
\newcommand{\blmpair}{\Lambda}
\newcommand{\blmpairr}{\Lambda^*}

\DeclareMathOperator{\Det}{Det}
\newcommand{\Detrc}{\Det_\R^\charge}
\newcommand{\Detrpc}{\Det_{\Rp}^{\charge}}
\newcommand{\sewisoloops}[1]{\operatorname{S}_{#1}}

\newcommand{\cftspec}{\mathcal{S}}

\newcommand{\lkenergy}{S}
\newcommand{\foliation}{\eta}
\newcommand{\inversion}{j}

\title{Two-loop Loewner potentials}
\author{
Yan Luo\thanks{Institut des Hautes Études Scientifiques, Bures-sur-Yvette, France.
\protect\url{luoy@ihes.fr}}
\; and \,
Sid Maibach\thanks{Institute for Applied Mathematics, University of Bonn, Germany.  
\protect\url{maibach@uni-bonn.de}}
}
\date{November 4, 2024}

\begin{document}
\maketitle

\begin{abstract}
We study a generalization of the Schramm--Loewner evolution loop measure to pairs of non-intersecting Jordan curves on the Riemann sphere.
We also introduce four equivalent definitions for a two-loop Loewner potential:
respectively expressing it in terms of normalized Brownian loop measure, zeta-regularized determinants of the Laplacian, an integral formula generalizing universal Liouville action, and Loewner--Kufarev energy of a foliation.
Moreover, we prove that the potential is finite if and only if both loops are Weil--Petersson quasicircles, that it is an Onsager--Machlup functional for the two-loop SLE, and a variational formula involving Schwarzian derivatives.

Addressing the question of minimization of the two-loop Loewner potential, we find that any such minimizers must be pairs of circles.
However, the potential is not bounded, diverging to negative infinity as the circles move away from each other and to positive infinity as the circles merge, thus preventing a definition of two-loop Loewner energy for the prospective large deviations principle for the two-loop SLE.

To remedy the divergence, we study a way of generalizing the two-loop Loewner potential by taking into account how conformal field theory (CFT) partition functions depend on the modulus of the annulus between the loops.
This generalization is motivated by the correspondence between SLE and CFT, and it also emerges from the geometry of the real determinant line bundle as introduced by Kontsevich and Suhov.
\end{abstract}

\newpage

\tableofcontents

\newpage

\section{Introduction}

\subsection{Main results about two-loop Loewner potentials}

Much of the conformal geometry of a simple loop $\gamma$ in the Riemann sphere $\hat \C$ may be described by a Möbius-invariant, real-valued function $\lenergy{\hat \C}(\gamma)$ introduced by Takhtajan and Teo \cite{takhtajan_teo} under the name ``universal Liouville action'' as a Kähler potential for the Weil--Petersson metric on universal Teichmüller space, and by Rohde and Wang \cite{rohde_wang, wang_equivalent}
as ``Loewner energy'' in the context of Schramm--Loewner evolution (SLE) loop measures.
Subsequently, this function has appeared in various settings \cite{bishop_wpbeta, johansson_strong_szego, w_volume, renormalized_volume, johansson_viklund_coulomb_gas, sung_wang, maibach_peltola, viklund_wang_lk_energy, moving_frames}.
Following an expression in terms of zeta-regularized determinants of Laplacians, see \cite{wang_equivalent, peltola_wang} and Appendix~\ref{appendix:background}, the Loewner potential $\lpot{\hat \C}(\gamma)$ is a function such that $\lenergy{\hat \C}(\gamma)$ is obtained by normalization,
\begin{equation}
    \label{eq:lenergy_one}
    \lenergy{\hat \C}(\gamma) = 12\big(\lpot{\hat \C}(\gamma) - \inf_{\eta} \lpot{\hat \C}(\eta)\big), \qquad
    \lpot{\hat \C}(\gamma) =
    \log \frac{
    \detz{\restrict{g}{\hat \C}}
    }{
    \detz{\restrict{g}{D_1}}
    \detz{\restrict{g}{D_2}}
    },
\end{equation}
where $D_1$ and $D_2$ are the domains in the complement of $\gamma$ and the infimum taken over all simple loops is realized by the loop being a circle.

In this paper, we consider two non-intersecting simple loops $\gamma_1$, $\gamma_2$ in $\hat \C$, see Figure~\ref{fig:setup}.
We introduce four formulas for the Loewner potential $\lpot{\hat \C, 2}(\gamma_1, \gamma_2)$ of the pair of loops and prove their equivalence up to constants for smooth loops.
\begin{enumerate}
    \item
        The initial definition in Section~\ref{section:two_loop_sle} combines the one-loop Loewner potential of each individual loop with an interaction term,
        \begin{equation}
            \label{eq:lpot_def}
            \lpot{\hat \C, 2}(\gamma_1, \gamma_2) = \lpot{\hat \C}(\gamma_1) + \lpot{\hat \C}(\gamma_2) + \blmpairr(\gamma_1, \gamma_2).
        \end{equation}
        The interaction through the normalized Brownian loop measure $\blmpairr$, see \cite{loop_soup, field_lawler} and Appendix~\ref{appendix:blm}, is motivated by the construction of a two-loop SLE measure and the fact that, by Theorem~\ref{thm:onsager_machlup_two}, this two-loop Loewner potential is an Onsager--Machlup functional for said SLE measure, analogous to the main result in \cite{carfagnini_wang} for one loop.
    \item
        In Section~\ref{section:lpot_detz}, we prove the formulation in terms of zeta-regularized determinants of Laplacians (see Appendix~\ref{appendix:zeta}) 
        as it is suggested in \cite{peltola_wang},
        \begin{equation}
        \label{eq:lpot_two_detz}
            \lpot{\hat \C, 2}(\gamma_1, \gamma_2)
            =
            \log \frac{
            \detz{\restrict{g}{\hat \C}}
            }{
            \detz{\restrict{g}{D_1}}
            \detz{\restrict{g}{A}}
            \detz{\restrict{g}{D_2}}
            }
            + \const
            ,
        \end{equation}
        where the setup is as in Figure~\ref{fig:setup} and the constant is independent of the loops.
    \item 
        In Section~\ref{section:pre_schwarzians}, we find a generalization of the universal Liouville action \cite{takhtajan_teo}, involving the pre-Schwarzians $\preschwarzian{f} = f''/f'$ of conformal maps uniformizing the complement of the pair of  loops as in Figure~\ref{fig:setup}. Theorem~\ref{thm:lpot_liouville} states that
        \begin{equation}
        \label{eq:lpot_preschwarzian}
        \begin{aligned}
            \lpot{\hat \C, 2}(\gamma_1, \gamma_2)
            =\; &
            \lpot{\hat \C, 2}(e^{-2\pi \tau} S^1, S^1)
            \\
            &+ \frac{1}{12\pi} \bigg(
            \;
                \iint\limits_{e^{-2 \pi \tau} \disk} \big|
                \preschwarzian{f_1}
                \big|^2
                \;
                \dz
                +
                \iint\limits_{\annulus_\tau} \big|
                \preschwarzian{f_A}
                \big|^2
                \;
                \dz
                +
                \iint\limits_{\disk^*} \big|
                \preschwarzian{f_2}
                \big|^2
                \;
                \dz
            \bigg)
            \\
            &- \frac{1}{3} \log \left|
                \frac{f_2'(\infty)}{f_1'(0)}
            \right|
            .
        \end{aligned}
        \end{equation}
    \item 
        In Section~\ref{section:foliation_energy}, we find a relation to the Loewner--Kufarev energy of a measure~$\rho$, denoted $\lkenergy(\rho)$ and first defined in \cite{ang_park_wang_radial}.
        The measure is obtained through the Loewner--Kufarev equation of the foliation of $\C \setminus \{0\}$ by equipotential loops as obtained from the uniformizing maps in Figures~\ref{fig:setup} and~\ref{fig:foliation}.
        We prove the following generalization of the results in \cite{viklund_wang_lk_energy} for the one-loop Loewner energy,
        \begin{equation}
            \label{eq:lpot_foliation_energy}
            \lpot{\hat \C, 2}(\gamma_1, \gamma_2)
            =
            \lpot{\hat \C, 2}(e^{-2\pi\tau} S^1, S^1)
            + \frac{4}{3}\lkenergy(\rho)
            - \frac{1}{6} \log \left|
                \frac{
                    f_2'(\infty)
                }{
                    f_1'(0)
                }
            \right|
            + \const
            .
        \end{equation}
\end{enumerate}
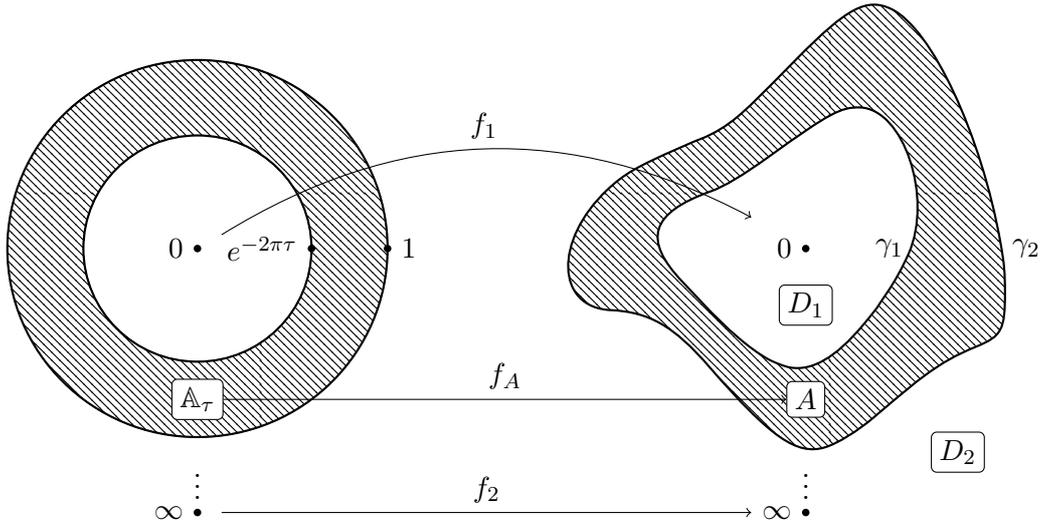
\begin{figure}
\centering
\begin{tikzpicture}
    \definecolor{color2}{HTML}{F1A340}
    \definecolor{color1}{HTML}{998EC3}

    \node (pic_left) at (-4, 0) {};
    \begin{scope}[shift=(pic_left)]
        \draw[fill, pattern=north west lines, even odd rule]
            (0,0) ellipse (2.5 and 2.5) ellipse (1.5 and 1.5);
        \draw[thick]
            (0,0) ellipse (1.5 and 1.5);
        \draw[thick]
            (0,0) ellipse (2.5 and 2.5);
        \node[circle, fill, inner sep=1pt, outer sep=0.0cm, label=left:{$0$}]
            (zero_left) at (0, 0) {};
        \node[circle, fill, inner sep=1pt, outer sep=0.0cm, label=right:{$1$}]
            (one_left) at (2.5, 0) {};
        \node[circle, fill, inner sep=1pt, outer sep=0.0cm, label=left:{$e^{- 2 \pi \tau}$}]
            (modulus) at (1.5, 0) {};
        \node[inner sep=0pt, outer sep=0.3cm]
            (D_one_left) at (0, 0) {};
        \node[circle, fill, inner sep=1pt, outer sep=0.0cm, label=left:{$\infty$}, label=above:{$\vdots$}]
            (infty_left) at (0, -3.5) {};
        \node[circle, inner sep=0pt, outer sep=0.3cm]
            (D_two_left) at (0, -3.5) {};
        \node[draw, fill=white, rounded corners=2pt, inner sep=3pt]
        %\node[]
            (A_left) at (0, -2) {$\annulus_\tau$};
    \end{scope}

    \node (pic_right) at (4, 0) {};
    \begin{scope}[shift=(pic_right)]
        \begin{scope}[scale=0.65, shift={(-9.6,-21.7)}]
            \path[draw, thick, pattern=north west lines] (9.9946, 17.6467).. controls (10.2477, 17.7261) and (10.6438, 17.972) .. (10.9746, 18.2552).. controls (11.6131, 18.802) and (11.7842, 18.9437) .. (11.9988, 19.1031).. controls (12.3714, 19.38) and (12.6244, 19.5231) .. (13.0638, 19.7057).. controls (13.5627, 19.913) and (13.6325, 20.0466) .. (13.6307, 20.7905).. controls (13.6297, 21.2372) and (13.6166, 21.385) .. (13.5395, 21.8224).. controls (13.1523, 24.020) and (12.1505, 26.1833) .. (11.3376, 26.5835).. controls (10.839, 26.829) and (10.4436, 26.6201) .. (9.4142, 25.5672).. controls (8.5985, 24.7329) and (8.4551, 24.5913) .. (8.2257, 24.3948).. controls (7.995, 24.1971) and (7.7374, 24.0394) .. (7.4811, 23.939).. controls (7.401, 23.9077) and (7.2998, 23.866) .. (7.2562, 23.8464).. controls (7.2125, 23.8268) and (7.1053, 23.7792) .. (7.018, 23.7408).. controls (6.1834, 23.373) and (5.9235, 23.2205) .. (5.5855, 22.8999).. controls (4.4974, 21.8678) and (4.5365, 20.4967) .. (5.6554, 20.4436).. controls (6.4341, 20.4066) and (6.9355, 20.1166) .. (7.6411, 19.2951).. controls (8.8884, 17.8428) and (9.4238, 17.4678) .. (9.9946, 17.6467) -- cycle(9.2802, 19.2745).. controls (9.1817, 19.2913) and (8.959, 19.3764) .. (8.8414, 19.4421).. controls (8.4755, 19.6465) and (8.014, 20.0441) .. (7.5065, 20.5921).. controls (7.3852, 20.723) and (7.2073, 20.9135) .. (7.1111, 21.0154).. controls (6.2892, 21.886) and (6.4543, 22.4242) .. (7.6795, 22.8678).. controls (7.9578, 22.9685) and (8.3253, 23.1908) .. (8.8835, 23.5961).. controls (10.0738, 24.4603) and (10.3143, 24.5907) .. (10.6957, 24.5785).. controls (11.5803, 24.5501) and (12.1794, 22.7374) .. (11.6811, 21.5975).. controls (11.56, 21.3206) and (11.4035, 21.0335) .. (11.2509, 20.8085).. controls (11.1638, 20.6802) and (11.0807, 20.5547) .. (11.0662, 20.5297).. controls (11.0417, 20.4875) and (10.9212, 20.3347) .. (10.7619, 20.1435).. controls (10.202, 19.4718) and (9.7333, 19.1969) .. (9.2802, 19.2745) -- cycle;
        \end{scope} 
        \node[circle, fill, inner sep=1pt, outer sep=0.0cm, label=left:{$0$}]
            (zero_right) at (0, 0) {};
        \node[]
            (gamma_1) at (1.1, 0) {$\gamma_1$};
        \node[]
            (gamma_2) at (2.9, 0) {$\gamma_2$};
        \node[inner sep=0pt, outer sep=0.7cm]
            (D_one_right) at (0, 0) {};
        \node[circle, fill, inner sep=1pt, outer sep=0.0cm, label=left:{$\infty$}, label=above:{$\vdots$}]
            (infty_right) at (0, -3.5) {};
        \node[inner sep=0pt, outer sep=0.7cm]
            (D_two_right) at (0, -3.5) {};
        \node[draw, fill=white, rounded corners=2pt, inner sep=3pt]
        %\node[]
            (A_right) at (0, -2) {$A$};
        \node[draw, fill=white, rounded corners=2pt, inner sep=3pt]
            (D_one_label) at (0, -0.75) {$D_1$};
        \node[draw, fill=white, rounded corners=2pt, inner sep=3pt]
            (D_two_label) at (2, -2.7) {$D_2$};
    \end{scope}

    \draw[->, bend left] (D_one_left) edge[above] node {$f_1$} (D_one_right);
    \draw[->] (A_left) edge[above] node {$f_A$} (A_right);
    \draw[->] (D_two_left) edge[above] node {$f_2$} (D_two_right);
    
\end{tikzpicture}
\caption{
    In this work, we study the conformal geometry of two non-intersecting simple smooth loops $\gamma_1$, $\gamma_2$ in the Riemann sphere $\hat \C$. By conformal invariance, let the disk $D_1$ enclosed by $\gamma_1$ contain $0$ and let the disk $D_2$ enclosed by $\gamma_2$ contain $\infty$.
    $A$ is the annulus between $\gamma_1$ and $\gamma_2$ and we denote the modulus of $A$ by $\tau$, and the uniformizing map from the standard annulus $\annulus_\tau = \setsuchthat{z \in \C}{e^{-2 \pi \tau} < |z| < 1}$ to $A$ by $f_A$.
    Also, $f_1$ is the Riemann mapping from $e^{-2 \pi \tau}\disk$ to $D_1$ such that $f_1(0) = 0$ and $f_1'(0) > 0$ and $f_2$ is the Riemann mapping from $\disk^* = \hat \C \setminus \disk$ to $D_2$ such that $f_2(\infty) = \infty$ and $f_2'(\infty) > 0$, where $\disk = \setsuchthat{z \in \C}{|z| < 1}$ is the unit disk.
}
\label{fig:setup}
\end{figure}

On the subject of a natural regularity class for the loops, we find that $\lpot{\hat \C, 2}(\gamma_1, \gamma_2)$, as defined by Equation~\eqref{eq:lpot_def}, is finite if and only if both $\gamma_1$ and $\gamma_2$ are Weil--Petersson quasicircles.
This follows from the fact that, since the universal Liouville action of a single loop is finite if and only if it is a Weil--Petersson quasicircle \cite{takhtajan_teo}, the same holds true for the one-loop Loewner potential and, moreover, the Brownian loop measure is finite for any two disjoint non-polar sets \cite{field_lawler}. 
Alternatively, by \cite[Theorem~1.1]{viklund_wang_lk_energy}, finiteness of the formulas~\eqref{eq:lpot_preschwarzian} and \eqref{eq:lpot_foliation_energy} implies that the loops are Weil--Petersson quasicircles.
Let us mention that Weil--Petersson quasicircles have recently attracted a great deal of attention for its various equivalent descriptions and its important role in geometric function theory and conformal field theory \cite{analytical_foundation_cft, shen_wp, bishop_wpbeta}.

Next, in Section~\ref{section:variations_sle}, we address the question whether minimizers of the two-loop Loewner potential exist.
On the one hand, this is relevant for the definition of a notion of a two-loop Loewner energy by subtracting the infimum as in Equation~\eqref{eq:lenergy_one} for the one-loop Loewner potential.
Generally, notions of Loewner energy are expected to provide rate functions for the large deviations principle of SLE \cite{wang_survey, peltola_wang, guskov, abuzaid_healey_peltola}.
On the other hand, the task of finding the minimizers pertains to the question in conformal geometry of how to canonically embed various configurations of curves and graphs \cite{bonk_eremenko, piecewise_geodesic, peltola_wang}.
We employ a two-step strategy to identify the infimum of $\lpot{\hat \C, 2}(\gamma_1, \gamma_2)$.
\begin{enumerate}
    \item 
    We prove a variational formula for the two-loop Loewner potential, similar to that in \cite{takhtajan_teo, sung_wang} in the case of a single loop --- see Proposition~\ref{prop:two_loop_variation}.
    The formula allows for deformations of the loops by infinitesimal Beltrami differentials $\nu$ with support in $D_1 \cup D_2$, hence preserving the modulus of $A$,
    \begin{equation}
        \eval{\pdv{\varepsilon}}_{\varepsilon = 0}
        \lpot{\hat \C, 2}(\omega^{\varepsilon \nu}(\gamma_1), \omega^{\varepsilon \nu}(\gamma_2))
        = - \frac{1}{3\pi} \left(
            \iint_{D_1} \nu \, \schwarzian{f_1^{-1}} \, \dz
            + \iint_{D_2} \nu \, \schwarzian{f_2^{-1}} \, \dz
        \right).
    \end{equation}
    Note how the Schwarzian derivatives $\schwarzian{f_1^{-1}}$ and $\schwarzian{f_2^{-1}}$ vanish if and only if the uniformizing maps $f_1$ and $f_2$ as depicted in Figure~\ref{fig:setup} are Möbius transformations.
    Thus, if a minimizer exists, both loops must be circles.
    \item
    By conformal invariance, a configuration consisting of two circles is characterized by the modulus $\tau$ of the annulus bounded by the circles.
    Using the formulas for the zeta-regularized determinant of the Laplacian on a disk or annulus found in \cite{weisberger}, we explicitly compute $\lpot{\smash{\hat \C, 2}}(e^{-2\pi \tau} S^1, S^1)$.
    However, we find that the two-loop Loewner potential diverges to $-\infty$ as two circles move further apart in the limit $\tau \to \infty$.
    Moreover, it diverges to $+\infty$ as the circles merge in the limit $\tau \to 0$. See also the plot in Figure~\ref{fig:zeta_det}.
\end{enumerate}
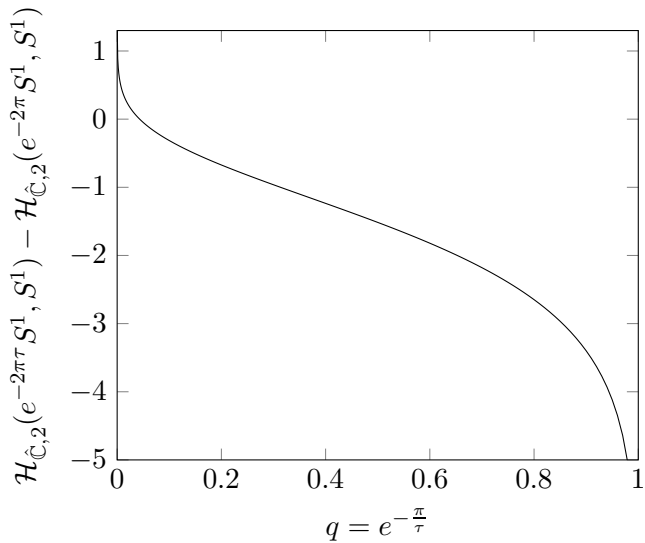
\begin{figure}
\centering
\begin{tikzpicture}

    \begin{axis}[
        xlabel={$q = e^{- \frac{\pi}{\tau}}$},
        ylabel={$\lpot{\hat \C, 2}(e^{-2\pi \tau} S^1, S^1) - \lpot{\hat \C, 2}(e^{-2\pi} S^1, S^1)$},
        xmin=0, xmax=1,
        ymin=-5, ymax=1.3,
        xtick={0, 0.2, 0.4, 0.6, 0.8, 1},
        ytick={-5, -4, -3, -2, -1, 0, 1},
    ]
    
    \addplot[]
        coordinates {
            (0.000100, 1.352715)
            (0.001100, 0.890204)
            (0.002100, 0.757117)
            (0.003100, 0.674229)
            (0.004100, 0.613162)
            (0.005100, 0.564431)
            (0.006100, 0.523670)
            (0.007100, 0.488498)
            (0.008100, 0.457471)
            (0.009100, 0.429642)
            (0.010100, 0.404362)
            (0.011100, 0.381161)
            (0.012100, 0.359689)
            (0.013100, 0.339680)
            (0.014100, 0.320924)
            (0.015100, 0.303255)
            (0.016100, 0.286538)
            (0.017100, 0.270660)
            (0.018100, 0.255531)
            (0.019100, 0.241072)
            (0.020100, 0.227217)
            (0.021100, 0.213910)
            (0.022100, 0.201101)
            (0.023100, 0.188747)
            (0.024100, 0.176813)
            (0.025100, 0.165264)
            (0.026100, 0.154072)
            (0.027100, 0.143210)
            (0.028100, 0.132657)
            (0.029100, 0.122391)
            (0.030100, 0.112393)
            (0.031100, 0.102647)
            (0.032100, 0.093137)
            (0.033100, 0.083849)
            (0.034100, 0.074772)
            (0.035100, 0.065892)
            (0.036100, 0.057199)
            (0.037100, 0.048684)
            (0.038100, 0.040338)
            (0.039100, 0.032151)
            (0.040100, 0.024117)
            (0.041100, 0.016228)
            (0.042100, 0.008478)
            (0.043100, 0.000860)
            (0.044100, -0.006632)
            (0.045100, -0.014003)
            (0.046100, -0.021258)
            (0.047100, -0.028402)
            (0.048100, -0.035439)
            (0.049100, -0.042373)
            (0.050100, -0.049210)
            (0.051100, -0.055951)
            (0.052100, -0.062601)
            (0.053100, -0.069163)
            (0.054100, -0.075641)
            (0.055100, -0.082036)
            (0.056100, -0.088353)
            (0.057100, -0.094593)
            (0.058100, -0.100759)
            (0.059100, -0.106854)
            (0.060100, -0.112880)
            (0.061100, -0.118839)
            (0.062100, -0.124733)
            (0.063100, -0.130565)
            (0.064100, -0.136335)
            (0.065100, -0.142046)
            (0.066100, -0.147700)
            (0.067100, -0.153298)
            (0.068100, -0.158842)
            (0.069100, -0.164333)
            (0.070100, -0.169773)
            (0.071100, -0.175163)
            (0.072100, -0.180505)
            (0.073100, -0.185799)
            (0.074100, -0.191048)
            (0.075100, -0.196251)
            (0.076100, -0.201411)
            (0.077100, -0.206529)
            (0.078100, -0.211605)
            (0.079100, -0.216640)
            (0.080100, -0.221636)
            (0.081100, -0.226593)
            (0.082100, -0.231513)
            (0.083100, -0.236396)
            (0.084100, -0.241243)
            (0.085100, -0.246055)
            (0.086100, -0.250832)
            (0.087100, -0.255576)
            (0.088100, -0.260287)
            (0.089100, -0.264966)
            (0.090100, -0.269614)
            (0.091100, -0.274231)
            (0.092100, -0.278817)
            (0.093100, -0.283375)
            (0.094100, -0.287903)
            (0.095100, -0.292403)
            (0.096100, -0.296876)
            (0.097100, -0.301321)
            (0.098100, -0.305740)
            (0.099100, -0.310133)
            (0.100000, -0.314064)
            (0.110000, -0.356455)
            (0.120000, -0.396754)
            (0.130000, -0.435292)
            (0.140000, -0.472330)
            (0.150000, -0.508078)
            (0.160000, -0.542708)
            (0.170000, -0.576362)
            (0.180000, -0.609160)
            (0.190000, -0.641203)
            (0.200000, -0.672581)
            (0.210000, -0.703368)
            (0.220000, -0.733631)
            (0.230000, -0.763430)
            (0.240000, -0.792817)
            (0.250000, -0.821840)
            (0.260000, -0.850540)
            (0.270000, -0.878957)
            (0.280000, -0.907126)
            (0.290000, -0.935079)
            (0.300000, -0.962848)
            (0.310000, -0.990461)
            (0.320000, -1.017943)
            (0.330000, -1.045321)
            (0.340000, -1.072617)
            (0.350000, -1.099854)
            (0.360000, -1.127055)
            (0.370000, -1.154240)
            (0.380000, -1.181428)
            (0.390000, -1.208641)
            (0.400000, -1.235897)
            (0.410000, -1.263215)
            (0.420000, -1.290614)
            (0.430000, -1.318113)
            (0.440000, -1.345731)
            (0.450000, -1.373486)
            (0.460000, -1.401397)
            (0.470000, -1.429483)
            (0.480000, -1.457763)
            (0.490000, -1.486259)
            (0.500000, -1.514988)
            (0.510000, -1.543973)
            (0.520000, -1.573236)
            (0.530000, -1.602797)
            (0.540000, -1.632681)
            (0.550000, -1.662912)
            (0.560000, -1.693515)
            (0.570000, -1.724517)
            (0.580000, -1.755945)
            (0.590000, -1.787830)
            (0.600000, -1.820202)
            (0.610000, -1.853095)
            (0.620000, -1.886545)
            (0.630000, -1.920589)
            (0.640000, -1.955268)
            (0.650000, -1.990626)
            (0.660000, -2.026711)
            (0.670000, -2.063573)
            (0.680000, -2.101268)
            (0.690000, -2.139857)
            (0.700000, -2.179406)
            (0.710000, -2.219987)
            (0.720000, -2.261681)
            (0.730000, -2.304577)
            (0.740000, -2.348771)
            (0.750000, -2.394375)
            (0.760000, -2.441509)
            (0.770000, -2.490314)
            (0.780000, -2.540943)
            (0.790000, -2.593576)
            (0.800000, -2.648415)
            (0.810000, -2.705695)
            (0.820000, -2.765689)
            (0.830000, -2.828713)
            (0.840000, -2.895146)
            (0.850000, -2.965436)
            (0.860000, -3.040124)
            (0.870000, -3.119873)
            (0.880000, -3.205503)
            (0.890000, -3.298049)
            (0.900000, -3.398843)
            (0.910000, -3.509636)
            (0.920000, -3.632803)
            (0.930000, -3.771669)
            (0.940000, -3.931108)
            (0.950000, -4.118670)
            (0.960000, -4.347010)
            (0.970000, -4.639842)
            (0.980000, -5.050414)
            (0.990000, -5.748624)
        };
    \end{axis}
\end{tikzpicture}
\caption{
    Plot of the two-loop Loewner potential of a pair of circles separated by an annulus with modulus $\tau$, normalized by the value at $\tau = 1$.
    Note the divergences as $\tau \to \infty$, $q \to 1$, moving the circles apart, and as $\tau \to 0$, $q \to 0$, merging the circles.
}
\label{fig:zeta_det}
\end{figure}

As a final point, in Section~\ref{section:cft_partition_functions}, we provide a remedy to the problem of non-existence of minimizers for the two-loop potential described above.
As we will explain below, the solution is motivated by the role that SLE measures play as the universal laws of interfaces in
conformal field theory.
For now, let us just give a precise mathematical description of the procedure.
After fixing a constant $\charge \neq 0$, we will pick positive real numbers ${Z_g(D) > 0}$ indexed by $D$ being either $\hat \C$, or any simply connected domain, or any doubly-connected domain in $\hat \C$, and subsequently by $g$ being any smooth metric on $D$.
Assume that the family $Z_g(D)$ satisfies the following properties under conformal transformations and diffeomorphisms,
\begin{equation}
\begin{gathered}
    Z_{e^{2\sigma} g}(D) = e^{\charge \conformalanomaly(\sigma, g)} Z_g(D), \qquad
    Z_{f_* g}(f(D)) = Z_g(D),
    \\
    \conformalanomaly(\sigma, g) =
    \frac{1}{12 \pi} \iint_\Sigma \bigg(
    \frac{1}{2} |\nabla_g \sigma|_g^2 + R_g \sigma
    \bigg) \dvol{g}
    + \frac{1}{12 \pi} \int_{\partial_\Sigma} k_g \sigma \, \dboundary{g},
\end{gathered}
\end{equation}
where $\sigma : D \to \R$ is any smooth function and $f : D \to \hat \C$ is a diffeomorphism onto its image.
$\nabla_g$, $R_g$, and $k_g$ are respectively the divergence, Gaussian curvature, and boundary curvature with respect to the metric $g$. We define a generalized notion of two-loop Loewner potential by
\begin{equation}
    \label{eq:def_lpot_z_two}
    \lpotx{\hat \C, 2}{Z}(\gamma_1, \gamma_2) =
    \frac{2}{\charge}
    \log \frac{
        Z_g(D_1)
        Z_g(A)
        Z_g(D_2)
    }{
        Z_g(\hat \C)
    }
    .
\end{equation}
As explained in Section~\ref{section:trivialization_sle}, and as already observed in \cite{maibach_peltola} for a single loop, we recover the probabilistic two-loop Loewner potential $\lpotx{\hat \C, 2}{Z}(\gamma_1, \gamma_2) = \lpot{\hat \C, 2}(\gamma_1, \gamma_2) + \const$ in Equation~\eqref{eq:lpot_two_detz} up to a constant by setting
\begin{equation}
    Z_g(D) = \left(\frac{
        \detz{\restrict{g}{D}}
    }{
        \boundaryterm{g}(D)
    }\right)^{-\charge/2}
    ,
    \qquad
    \boundaryterm{g}(\Sigma) = \exp\bigg(\frac{1}{4\pi} \int_{\partial_\Sigma} k_g \dboundary{g}\bigg).
\end{equation}
Generally, by the uniformization theorem, any family $Z_g(D)$ is completely determined by the values of
\begin{equation}
    Z_g(\hat \C), \qquad
    Z_{\flatmetric}(\disk), \qquad
    Z_{\flatmetric}(\annulus_\tau), \tau > 0,
\end{equation}
where $g$ belongs to the conformal class of $\hat \C$, and $\flatmetric = \dd x^2 + \dd y^2$ is the flat metric.
In particular, an analogous generalization of the one-loop Loewner potential can only be changed by a finite constant with respect to Equation~\eqref{eq:lenergy_one}.
Assuming that $Z_{\flatmetric}(\annulus_\tau)$ is a smooth function of $\tau$, we conclude with Proposition~\ref{prop:lpot_z_circles}, which provides the condition that there exists a pair of disjoint circles attaining the infimum of new two-loop Loewner potential $\lpotx{\hat \C, 2}{Z}$ if and only if the function
\begin{equation}
\label{eq:criterion}
    e^{-\frac{\pi}{3} \charge \tau}
    Z_{\flatmetric}(
        \annulus_\tau
    )
\end{equation}
has a global minimum in $\tau \in (0, \infty)$.

Finally, we would like to highlight two technical contributions in this work.
Namely, in Section~\ref{section:foliation_energy}, we prove a generalization of the Grunsky inequality (Proposition~\ref{prop:multiple_grunsky}), which is a preliminary to obtain Equation~\eqref{eq:lpot_foliation_energy}.
Secondly, Lemma~\ref{lemma:det_lim} in Appendix~\ref{appendix:background} is a novel renormalization formula for zeta-regularized determinants of Laplacians analogous to normalized Brownian loop measure. The lemma is applied to prove Equation~\eqref{eq:lpot_two_detz}.

\subsection{Motivation of the CFT definition of Loewner potentials}

In the following, we motivate the generalized definitions~\eqref{eq:def_lpot_z_two} of two-loop Loewner potentials $\lpotx{\hat \C, 2}{Z}(\gamma_1, \gamma_2)$ from the point of view of (Euclidean) 2D conformal field theory (CFT).
For mathematical introductions to the latter, see for instance \cite{gawedzki_lectures, schottenloher, lcft_review}.
In the CFT terminology, the constant $\charge \neq 0$ is called the central charge and the numbers $Z_g(D) > 0$ are the partition functions of any fixed CFT of central charge $\charge$ in the domains $D$ with respect to the metrics $g$.

Generally speaking, one may think of the loops $\gamma_1$, $\gamma_2$ as interfaces arising in configurations of a fixed CFT and the Loewner potential as an action functional for such interfaces.
Let us explain this heuristically using a statistical mechanical system on a finite discretization of $\hat \C$ with states $\sigma$, a local energy function $S(\sigma)$, and inverse temperature $\beta$.
For example, consider the critical Ising model where states take values $+1$, $-1$ and $S$ considers a nearest neighbor interaction.
The partition function is given by
\begin{equation}
    Z(\hat \C) = \sum_\sigma e^{-\beta S(\sigma)}.
\end{equation}
When considering the partition function $Z(D)$ on a subdomain $D \subset \hat \C$ such as $D_1$, $A$, $D_2$ in Figure~\ref{fig:setup}, we have to fix boundary conditions on the restricted state $\restrict{\sigma}{D}$, say $+1$ or $-1$ in the Ising model, such that $\gamma_1$ and $\gamma_2$ are interfaces for any states with these boundary conditions.
The problem of fixing the boundary condition in conformally invariant ways is the starting point of Cardy's work on boundary conformal field theory
\cite{cardy_boundary_conditions, cardy_fusion, cardy_bcft} which we briefly come back to in Example~\ref{example:characters}.
With the setup above, the probability of $\gamma_1$, $\gamma_2$ appearing as interfaces becomes
\begin{equation}
    \quad \sum_{\mathclap{\tikz{
        \node[circle, fill, inner sep=1pt, outer sep=0.0cm] (A) at (0, 0) {};
        \node[draw, rounded corners=2pt, inner sep=3pt, anchor=north] (B) at (0, -0.3) {$\substack{\text{$\sigma$ such that} \\ \text{$\gamma_1$, $\gamma_2$ are interfaces}}$};
        \draw[] (A) -- (B);
    }}} \quad
    \frac{e^{-\beta S(\sigma)}}{Z(\hat \C)}
    =
    \quad \sum_{\mathclap{\tikz{
        \node[circle, fill, inner sep=1pt, outer sep=0.0cm] (A) at (0, 0) {};
        \node[draw, rounded corners=2pt, inner sep=3pt, anchor=north] (B) at (0, -0.3) {$\substack{\text{$\restrict{\sigma}{D_1}$, $\restrict{\sigma}{A}$, $\restrict{\sigma}{D_2}$} \\ \text{satisfying} \\ \text{boundary conditions}}$};
        \draw[] (A) -- (B);
    }}} \quad
    \frac{e^{-\beta \big(
        S(\restrict{\sigma}{D_1}) +
        S(\restrict{\sigma}{A}) +
        S(\restrict{\sigma}{D_2})
    \big)}}{Z(\hat \C)}
    =
    \frac{Z(D_1) Z(A) Z(D_2)}{Z(\hat \C)}.
\end{equation}
Letting the Loewner potential define a measure on pairs of loops with density $e^{\frac{\charge}{2}\lpotx{\hat \C, 2}{Z}(\gamma_1, \gamma_2)}$ leads to the definition of two-loop Loewner potential in Equation~\eqref{eq:def_lpot_z_two}.
Note that this argument may be generalized to arbitrary configurations of loops, leading to the notions of Loewner potential defined in Appendix~\ref{appendix:detrc}.

Returning to mathematics, Theorem~\ref{thm:onsager_machlup_two} stating that the probabilistic two-loop Loewner potential~\eqref{eq:lpot_two_detz} is an Onsager--Machlup functional for the two-loop SLE measures defined in Section~\ref{section:two_loop_sle} essentially confirms that the Loewner potential is what we called the ``density'' of a measure on pairs of loops, where this measure is said two-loop $\SLE_\kappa$ measure with
\begin{equation}
    \charge = \frac{(6 - \kappa)(3\kappa - 8)}{2 \kappa}, \qquad \kappa \in (0, 4].
\end{equation}

In probability theory, SLE as originally defined by Schramm \cite{schramm} in the chordal case has been generalized to loops in the Riemann sphere and general Riemann surfaces \cite{werner07, kontsevich_suhov, benoist_dubedat, kemppainen_werner, Zhan21, ang_cai_sun_wu}.
Axiomatic characterizations of chordal SLE \cite{schramm, lsw_conformal_restriction} and recently also of loop SLE \cite{kontsevich_suhov, baverez_jego, gordina_wei_wang} support the appearance of SLE as interfaces in CFT as conjectured in \cite{bauer_bernard_sle, friedrich_kalkkinen, kontsevich_suhov}.
Theorems about SLE as interfaces have been proven in several special cases of scaling limits of critical lattice models (for instance, see \cite{convergence_ising}).
Multiple SLE has been studied extensively in the chordal case, see \cite{lawler_multiple_chordal, peltola_wang, peltola_wu} and the references therein, and \cite{bauer_bernard_kytola} for the relation to CFT.
We would also like to point out the recent interest in chordal SLE in multiply-connected domains \cite{lawler_multiply_connected, aru_bordereau, alberts_byun_kang} and other works involving SLE loops in multiply-connected domains \cite{kassel_kenyon, cle_annuli}.
In the recent work \cite{baverez_jego}, multiple loop SLE is interpreted as the amplitude of a CFT-like theory associated to SLE.

Let us provide slightly more details towards the construction in Section~\ref{section:two_loop_sle}.
Assuming $\kappa \in (0, 4]$, we show that by starting with the product of two independent one-loop SLE measures
$\sleloop{\hat \C}$
as defined in \cite{Zhan21}, a two-loop SLE measure
\begin{equation}
    \label{eq:two_loop_sle}
    \sleloop{\hat \C, 2}(\dd \gamma_1, \dd \gamma_2) =
    \boldone_{\{\gamma_1 \cap \gamma_2 = \emptyset\}}
    e^{\frac{\charge}{2} \blmpairr(\gamma_1, \gamma_2)}
    \sleloop{\hat \C}(\dd \gamma_1)
    \otimes
    \sleloop{\hat \C}(\dd \gamma_2)
\end{equation}
is uniquely determined by conformal invariance and a cascade relation, which allows consecutive sampling of the two loops (See Theorem~\ref{thm:two_loop_measure}).
This definition involves the normalized Brownian loop measure $\blmpairr$, see \cite{loop_soup, field_lawler} and Appendix~\ref{appendix:blm}.

However, the above is by far not the only way to obtain a two-loop SLE measure.
It would be interesting to relate the following constructions to Equation~\eqref{eq:two_loop_sle}.
\begin{enumerate}
    \item 
    For $\kappa \in (8/3, 4]$, the one-loop measure $\sleloop{\hat \C}$ may also be defined from the counting measure over a conformal loop ensemble (CLE), see \cite{kemppainen_werner, ang_sun_integrability_v2, ang_cai_sun_wu}.
    In the same way, a two-loop SLE measure may be obtained by sampling two loops independently from the counting measure on CLE.
    \item
    Consider a single SLE loop $\gamma$ with $\kappa \in (4, 8)$ in $\hat \C \setminus \{0, \infty\}$.
    By the duality principle of SLE, \cite[Theorem~1.2]{ang_cai_sun_wu}, the inner and outer boundaries of $\gamma$ surrounding $0$ and $\infty$ forms a pair of disjoint $\SLE_{16/\kappa}$ loops, giving another two-loop SLE measure.
\end{enumerate}

Even though we have the universality result \cite[Theorem~1.9]{baverez_jego} of SLE loop measure in the sense of Malliavin--Kontsevich--Suhov \cite{malliavin_measure, kontsevich_suhov}, different instances of SLE loop measures with the same $\kappa$ might still be related by nontrivial Radon--Nikodym derivatives.
Since the complement of a single loop consists of simply connected domains, the Radon--Nikodym derivative in this case is merely a constant. However, when considering a two-loop measure, it may depend on the modulus $\tau$ of the annulus bounded by the pair of loops.
In Section~\ref{section:cft_partition_functions} and Appendix~\ref{appendix:detrc}, we recall how these Radon--Nikodym derivatives and also the partition functions $Z_g(\blank)$ may be encoded as trivializations of the real determinant line bundle \cite{kontsevich_suhov, benoist_dubedat, maibach_peltola}.
From this abstraction, the Equation~\eqref{eq:def_lpot_z_two} for the two-loop Loewner potential emerges quite naturally.

To conclude, this brings us to the conceptual insight of this work, clarifying the extent of universality of SLE random curves appearing as the law of interfaces in CFT. Namely, whenever interfaces separating a domain into multiply connected subdomains are studied, the conjectured law of the interface differs from the established probabilistic definition of SLE by a Radon-Nikodym derivative involving the partition functions that the CFT under consideration associates to the respective subdomains. These partition functions, in turn, are functions of the modular parameters of the domains.

\subsection*{Acknowledgments}
We would like to thank Eveliina Peltola and Yilin Wang for their input and feedback.
Y.L.~is supported by the European Union (ERC, RaConTeich, 101116694)\footnote{Views and opinions expressed are however those of the authors only and do not necessarily reflect those of the European Union or the European Research Council Executive Agency. Neither the European Union nor the granting authority can be held responsible for them.} and previously by the Fondation Sciences Mathématiques de Paris (FSMP).
S.M.~is supported by the Deutsche Forschungsgemeinschaft (DFG, German Research Foundation) under Germany's Excellence Strategy EXC-2047/1-390685813 and by the grant CRC 1060 ``The
Mathematics of Emergent Effects'' (Project-ID 211504053).
S.M.~acknowledges the support of the IHES (France), where parts of the work were completed.

\section{Two-loop SLE}
\label{section:two_loop_sle}

\subsection{Definition}

The two-loop $\SLE_\kappa$ measure $\sleloop{\hat \C, 2}$ should be absolutely continuous with respect to the product of two one-loop $\SLE_\kappa$ loop measures, 
\begin{equation}
    \sleloop{\hat \C, 2}(\dd \gamma_1, \dd \gamma_2)
    =
    \boldone_{\{\gamma_1 \cap \gamma_2 = \emptyset\}}
    e^{\frac{\charge}{2} \interaction(\gamma_1, \gamma_2)}
    \sleloop{\hat \C}(\dd \gamma_1)
    \otimes
    \sleloop{\hat \C}(\dd \gamma_2)
    ,
\end{equation}
where $\interaction$ is a measurable interaction term.
Besides the conformal invariance and finiteness properties also satisfied by the one-loop measure (see Appendix~\ref{appendix:one_loop_sle}), we expect a two-loop measure to satisfy the cascade relation
\begin{equation}
    \label{eq:cascade}
    \sleloop{\hat \C, 2}(\dd \gamma_1, \dd \gamma_2) = \sleloop{\hat \C}(\dd \gamma_1) \otimes
    \sum_{D \in \pi_0(\hat \C \setminus \gamma_1)}
    \sleloop{D}(\dd \gamma_2).
\end{equation}
In words, this means that the two-loop measure is the product of the one-loop measure in $\hat \C$ and a second loop which is sampled in the complement of the first loop.
We obtain the following result, where only consider simple loops, hence $\kappa \in (0, 4]$.
\begin{theorem}
\label{thm:two_loop_measure}
The two-loop $\SLE_\kappa$ measure for $\kappa \in (0, 4]$ on $\hat \C$ defined by
\begin{equation}
    \sleloop{\hat \C, 2}(\dd \gamma_1, \dd \gamma_2) =
    \boldone_{\{\gamma_1 \cap \gamma_2 = \emptyset\}}
    e^{\frac{\charge}{2} \blmpairr(\gamma_1, \gamma_2)}
    \sleloop{\hat \C}(\dd \gamma_1)
    \otimes
    \sleloop{\hat \C}(\dd \gamma_2)
\end{equation}
is the unique conformally invariant measure on pairs of non-intersecting loops in $\hat \C$, which is absolutely continuous with respect to the product measure of two single $\SLE_\kappa$ loop measures and satisfies the cascade relation \eqref{eq:cascade}.
\end{theorem}
\begin{proof}
    Conformal invariance is a direct consequence of the conformal invariance of one-loop measure $\sleloop{\hat \C}$ and the normalized Brownian loop measure in $\blmpairr(\blank, \blank)$, see respectively \cite{Zhan21, field_lawler}.
    By the restriction covariance property of the one-loop measure \cite{Zhan21}, see also Equation~\eqref{eq:restriction_covariance} in Appendix~\ref{appendix:blm}, the cascade relation \eqref{eq:cascade} determines the two-loop measure uniquely, since
    \begin{equation}
    \begin{aligned}
        \sleloop{\hat \C, 2}(\dd \gamma_1, \dd \gamma_2)
        &=
        \sleloop{\hat \C}(\dd \gamma_1) \otimes
        \sum_{D \in \pi_0(\hat \C \setminus \gamma_1)}
        \boldone_{\{\gamma_2 \subset D\}}
        e^{\frac{\charge}{2} \interaction(\gamma_1, \gamma_2)}
        \sleloop{\hat \C}(\dd \gamma_2)
        \\ &=
        \sleloop{\hat \C}(\dd \gamma_1) \otimes
        \sum_{D \in \pi_0(\hat \C \setminus \gamma_1)}
        \boldone_{\{\gamma_2 \subset D\}}
        e^{
            \frac{\charge}{2}\big(
                \interaction(\gamma_1, \gamma_2)
                - \blmpairr(\gamma_2, \hat \C \setminus D)
            \big)
        }
        \sleloop{D}(\dd \gamma_2)
    \end{aligned}
    \end{equation}
    implies that \eqref{eq:cascade} holds precisely if
    \begin{equation}
        \interaction(\gamma_1, \gamma_2) = \blmpairr(\gamma_1, \gamma_2).
    \end{equation}
\end{proof}

\subsection{Onsager--Machlup functional}

The interpretation of $\blmpairr(\gamma_1, \gamma_2)$ as the interaction energy of a pair of loops leads to the conformally invariant notion of Loewner potential\footnote{
    See Equation~\eqref{eq:lenergy_one}
    for the relation of Loewner potential and energy.
}
of a pair of loops in Equation~\eqref{eq:lpot_def},
\begin{equation}
    \label{eq:lpot_two_def_section}
    \lpot{\hat \C, 2}(\gamma_1, \gamma_2) = \lpot{\hat \C}(\gamma_1) + \lpot{\hat \C}(\gamma_2) + \blmpairr(\gamma_1, \gamma_2).
\end{equation}
Following the one-loop case presented in \cite{carfagnini_wang},
we show that $\lpot{\hat \C, 2}$ is an Onsager--Machlup functional for the two-loop measure.

Let $\gamma$ be a loop with analytic parametrization $f_\gamma : S^1 \to \gamma$. By analyticity, $f_\gamma$ extends to an annulus $U_R = \setsuchthat{z \in \C}{R < |z| < 1/R}$ for some $0<R<1$. Define the following family over $0 < \varepsilon < R$ of neighborhoods of loops close to $\gamma$,
\begin{equation}
\label{eq:loop_nbhd}
\loopnbhd{\varepsilon}{\gamma} = \setsuchthat{f_\gamma(\eta)}{
    \text{$\eta$ non-contractible loop in $U_\varepsilon$}
}.
\end{equation}
In particular, we have
\begin{equation}
\loopnbhd{\varepsilon}{S^1} = \{\text{non-contractible loops in $U_\varepsilon$}\}.
\end{equation}

\begin{theorem}[Onsager--Machlup functional of two-loop $\SLE_\kappa$]
\label{thm:onsager_machlup_two}
    For two pairs of non-intersecting analytic loops $(\gamma_1, \gamma_2)$ and $(\xi_1, \xi_2)$ in $\hat \C$, and families of neighborhoods $\loopnbhd{\varepsilon}{\blank}$ of loops as in Equation~\eqref{eq:loop_nbhd}, we have
    \begin{equation}
        \lim_{\varepsilon \to 0}
        \frac{\sleloop{\hat \C, 2}(
            \loopnbhd{\varepsilon}{\gamma_1} \times
            \loopnbhd{\varepsilon}{\gamma_2}
        )}{\sleloop{\hat \C, 2}(
            \loopnbhd{\varepsilon}{\xi_1} \times
            \loopnbhd{\varepsilon}{\xi_2}
        )
        }
        =
        e^{
        \frac{\charge}{2} \left(
            \lpot{\hat \C, 2}(\gamma_1, \gamma_2)
            - \lpot{\hat \C, 2}(\xi_1, \xi_2)
        \right)
        }
        .
    \end{equation}
\end{theorem}
\begin{proof}
    Let $R > 0$ be small enough such that the neighborhoods of loops near $\gamma_1$ and $\gamma_2$ exist and become disjoint, i.e.\ $(\cup \, \loopnbhd{R}{\gamma_1}) \cap (\cup \, \loopnbhd{R}{\gamma_2}) = \emptyset$.
    Then for $0 < \varepsilon < R$,
\begin{align*}
    &
    \sleloop{\hat \C, 2}(\loopnbhd{\varepsilon}{\gamma_1} \times \loopnbhd{\varepsilon}{\gamma_2})
    \\ = \; &
    \int_{\loopnbhd{\varepsilon}{\gamma_1}}
    \int_{\loopnbhd{\varepsilon}{\gamma_2}}
    e^{\frac{\charge}{2}
        \blmpairr(\eta_1, \eta_2)
    }
    \sleloop{\hat \C}(\dd \eta_2)
    \otimes 
    \sleloop{\hat \C}(\dd \eta_1)
    \\ = \; &
    \int_{\loopnbhd{\varepsilon}{\gamma_1}}
    \int_{\loopnbhd{\varepsilon}{\gamma_2}}
    e^{
    \begin{subarray}{l}
        \frac{\charge}{2}
        \big( 
        \blmpairr(\eta_1, \eta_2)
        \\
        \phantom{\frac{\charge}{2} \big(}
        - \Lambda^*(\eta_1, \hat \C \setminus f_{\gamma_1}(U_R))
        - \Lambda^*(\eta_2, \hat \C \setminus f_{\gamma_1}(U_R))
        \big)
    \end{subarray}
    }
    \sleloop{f_{\gamma_1}(U_R)}(\dd \eta_2)
    \otimes \sleloop{f_{\gamma_2}(U_R)}(\dd \eta_1)
    \\ = \; &
    \int_{\loopnbhd{\varepsilon}{S^1}}
    \int_{\loopnbhd{\varepsilon}{S^1}}
    e^{
    \begin{subarray}{l}
        \frac{\charge}{2}
        \big( \blmpairr(f_{\gamma_1}(\eta_1), f_{\gamma_2}(\eta_2))
        \\
        \phantom{\frac{\charge}{2} \big(}
        - \Lambda^*(f_{\gamma_1}(\eta_1), \hat \C \setminus f_{\gamma_1}(U_R))
        - \Lambda^*(f_{\gamma_2}(\eta_2), \hat \C \setminus f_{\gamma_1}(U_R))
        \big)
    \end{subarray}
    }
    \sleloop{U_R}(\dd \eta_2)
    \otimes \sleloop{U_R}(\dd \eta_1)
    \\ = \; &
    \int_{\loopnbhd{\varepsilon}{S^1}}
    \int_{\loopnbhd{\varepsilon}{S^1}}
    e^{
    \begin{subarray}{l}
        \frac{\charge}{2}
        \big( \blmpairr(f_{\gamma_1}(\eta_1), f_{\gamma_2}(\eta_2))
        \\
        \phantom{\frac{\charge}{2} \big(}
        + \Lambda^*(\eta_1, \hat \C \setminus U_R)
        - \Lambda^*(f_{\gamma_1}(\eta_1), \hat \C \setminus f_{\gamma_1}(U_R))
        \\
        \phantom{\frac{\charge}{2} \big(}
        + \Lambda^*(\eta_2, \hat \C \setminus U_R)
        - \Lambda^*(f_{\gamma_2}(\eta_2), \hat \C \setminus f_{\gamma_1}(U_R))
        \big)
    \end{subarray}
    }
    \sleloop{\hat \C}(\dd \eta_2)
    \otimes \sleloop{\hat \C}(\dd \eta_1)
\end{align*}
By Corollary~2.4, Lemma~3.2 and Equation~(1.6) in \cite{carfagnini_wang}, we have for $j = 1, 2$,
\begin{equation}
\begin{aligned}
    \lim_{\varepsilon \to 0}
    \sup_{\eta_j \in \loopnbhd{\varepsilon}{S^1}}
    \left(
        \Lambda^*(\eta_j, \hat \C \setminus U_R)
        - \Lambda^*(f_{\gamma_j}(\eta_j), \hat \C \setminus f_{\gamma_j}(U_R))
    \right))
    &=
    \frac{1}{12} \lenergy{\hat \C}(\gamma_j),
    \\
    \lim_{\varepsilon \to 0}
    \inf_{\eta_j \in \loopnbhd{\varepsilon}{\gamma_j}}
    \left(
        \Lambda^*(\eta_j, \hat \C \setminus U_R)
        - \Lambda^*(f_{\gamma_j}(\eta_j), \hat \C \setminus f_{\gamma_j}(U_R))
    \right))
    &=
    \frac{1}{12} \lenergy{\hat \C}(\gamma_j).
\end{aligned}
\end{equation}
Since the analogous statements hold true for the other pair $(\xi_1, \xi_2)$ as well,
we obtain the following limit
\begin{equation}
\begin{aligned}
    \lim_{\varepsilon \to 0}
    \frac{
        \sleloop{\hat \C, 2}(\loopnbhd{\varepsilon}{\gamma_1} \times \loopnbhd{\varepsilon}{\gamma_2})
    }{
        \sleloop{\hat \C, 2}(\loopnbhd{\varepsilon}{\xi_1} \times \loopnbhd{\varepsilon}{\xi_2})
    }
    &=
    e^{
    \frac{\charge}{2} \left(
        \blmpairr(\gamma_1, \gamma_2)
        - \blmpairr(\xi_1, \xi_2)
        + \lenergy{\hat \C}(\gamma_1)
        + \lenergy{\hat \C}(\gamma_2)
        - \lenergy{\hat \C}(\xi_1)
        - \lenergy{\hat \C}(\xi_2)
    \right)
    }
    \\ &=
    e^{
    \frac{\charge}{2} \left(
        \blmpairr(\gamma_1, \gamma_2)
        - \blmpairr(\xi_1, \xi_2)
        + \lpot{\hat \C}(\gamma_1)
        + \lpot{\hat \C}(\gamma_2)
        - \lpot{\hat \C}(\xi_1)
        - \lpot{\hat \C}(\xi_2)
    \right)
    }
    \\ &=
    e^{
    \frac{\charge}{2} \left(
        \lpot{\hat \C, 2}(\gamma_1, \gamma_2)
        - \lpot{\hat \C, 2}(\xi_1, \xi_2)
    \right)
    }
    .
\end{aligned}
\end{equation}
\end{proof}

\section{Characterizations of the two-loop Loewner potential}

In this section, we give three more formulas for the two-loop Loewner potential, which is initialy defined using renormalized Brownian loop measrue in Equation~\eqref{eq:lpot_two_def_section}.

\subsection{Zeta-regularized determinants of Laplacians}
\label{section:lpot_detz}

Using Lemma~\ref{lemma:det_lim}, the two-loop Loewner potential as defined in \eqref{eq:lpot_def} may be expressed in terms of zeta-regularized determinants of Laplacians.
Up to a constant, this is in agreement with the more general definition suggested in \cite[Equation~(1.15)]{peltola_wang}.
\begin{proposition}
    \label{prop:lpot_detz}
    Let $g$ be any smooth metric on $\hat \C$ and $\gamma_1$, $\gamma_2$ any two non-intersecting simple smooth loops.
    With the setup as in Figure~\ref{fig:setup}, we have
    \begin{align*}
        \lpot{\hat \C, 2}(\gamma_1, \gamma_2)
        &=
        \lpot{\hat \C}(\gamma_1) + \lpot{\hat \C}(\gamma_2) + \blmpairr(\gamma_1, \gamma_2)
        \\ &=
        \log \frac{
        \detz{\restrict{g}{\hat \C}}
        }{
        \detz{\restrict{g}{D_1}}
        \detz{\restrict{g}{A}}
        \detz{\restrict{g}{D_2}}
        }
        +
        \const
        .
    \end{align*}
    The constant is independent of the loops and the metric $g$.
\end{proposition}
\begin{proof}
    Using conformal invariance, without loss of generality, assume that $0 \in D_1$ and $\infty \in D_2$.
    By Theorem~4.26 in \cite{field_lawler}, Proposition~2.1 in \cite{dubedat_partition_functions_couplings} and Lemma~\ref{lemma:det_lim}, we have
    \begin{equation}
    \label{eq:blm_detz}
    \begin{aligned}
        &\blmpairr(\gamma_1, \gamma_2)
        = \lim_{R \to \infty} \big(
            \blmpair_{R \, \disk}(\gamma_1, \gamma_2) - \log \log R
        \big)
        \\ = \; &
        \lim_{R \to \infty} \Bigg(
        \log \frac{
            \detz{\restrict{g}{D_1}}
            \detz{\restrict{g}{A \cup D_2 \setminus R \, \disk^*}}
            \detz{\restrict{g}{D_2 \setminus R \, \disk^*}}
            \detz{\restrict{g}{D_1 \cup A}}
        }{
            \detz{\restrict{g}{R \, \disk}}
            \detz{\restrict{g}{D_1}}
            \detz{\restrict{g}{A}}
            \detz{\restrict{g}{D_2 \setminus R \, \disk^*}}
        }
        - \log \log R
        \Bigg)
        \\ = \; &
        \lim_{R \to \infty} \Bigg(
        \log \frac{
            \detz{\restrict{g}{A \cup D_2 \setminus R \, \disk^*}}
            \detz{\restrict{g}{D_1 \cup A}}
        }{
            \detz{\restrict{g}{R \, \disk}}
            \detz{\restrict{g}{A}}
        }
        - \log \log R
        \Bigg)
        \\ = \; &
        \log \frac{
            \detz{\restrict{g}{A \cup D_2}}
            \detz{\restrict{g}{D_1 \cup A}}
        }{
            \detz{\restrict{g}{\hat \C}}
            \detz{\restrict{g}{A}}
        }
        + \const.
    \end{aligned}
    \end{equation}
    The statement now follows from expressions~ \eqref{eq:lenergy_one} and~\eqref{eq:lpot_one_detz} in Appendix~\ref{appendix:background} of the Loewner potential of a single loop in terms of zeta-regularized determinants of Laplacians.
\end{proof}

\subsection{Pre-Schwarzian formula for the Loewner potential}
\label{section:pre_schwarzians}

In this section, we prove that the two-loop Loewner potential~\eqref{eq:lpot_def} of smooth loops $\gamma_1$, $\gamma_2$ can be expanded around the energy of two circles with relative modulus $\tau$ equal to that of the annulus $A$ enclosed by $\gamma_1$ and $\gamma_2$ and a term which integrates the pre-Schwarzians $\preschwarzian{f} = f''/f'$ of the conformal maps in Figure~\ref{fig:setup}.
This formula is analogous to an expression of the one-loop Loewner potential found in \cite{wang_equivalent}, relating it to the universal Liouville action, which is a Kähler potential on universal Teichmüller space \cite{takhtajan_teo}.

\begin{theorem}
\label{thm:lpot_liouville}
    For two non-intersecting simple smooth loops $\gamma_1$, $\gamma_2$ in $\hat \C$ with the domains and conformal maps set up as in Figure~\ref{fig:setup}, the two-loop Loewner potential~\eqref{eq:lpot_def} equals
    \begin{equation}
    \label{eq:lpot_liouville_thm}
    \begin{aligned}
        \lpot{\hat \C, 2}(\gamma_1, \gamma_2)
        = \; &
        \lpot{\hat \C, 2}(e^{-2\pi \tau} S^1, S^1)
        \\ &+ \frac{1}{12\pi} \bigg(
            \iint\limits_{e^{-2 \pi \tau} \disk} \big|
            \preschwarzian{f_1}
            \big|^2
            \;
            \dz
            +
            \iint\limits_{\annulus_\tau} \big|
            \preschwarzian{f_A}
            \big|^2
            \;
            \dz
            +
            \iint\limits_{\disk^*} \big|
            \preschwarzian{f_2}
            \big|^2
            \;
            \dz
        \bigg)
        \\ &- \frac{1}{3} \log \left|
            \frac{f_2'(\infty)}{f_1'(0)}
        \right|
        .
    \end{aligned}
    \end{equation}
\end{theorem}
\begin{proof}
    Fix a metric $h = e^{2 \psi} \flatmetric$, where $\psi$ has any smooth cut-off such that
    \begin{equation}
        \psi(z) = \begin{cases}
            \frac{1}{2} \log\left(
                \frac{4}{(1 + |z|^2)^2}
            \right)
            & \text{near $\infty$}
            \\
            0
            & \text{near $A \cup D_1 \cup \disk$}.
        \end{cases}
    \end{equation}
    This way, $h$ is the flat metric on most of the sphere, with a round cap at $\infty$.
    Now we take the pullback of $h$ along the three conformal maps $f_1$, $f_A$, $f_2$, finding
    \begin{alignat}{5}
        \label{eq:g1}
        &g_1
        &&= f_1^*(\restrict{h}{D_1})
        &&= e^{2 (\psi\circ f_1 + \log |f_1'|)}
        &&\restrict{\flatmetric}{e^{-2\pi \tau}\disk}
        &&= e^{2 \log |f_1'|}
        \, \restrict{\flatmetric}{e^{-2\pi \tau}\disk},
        \\
        \label{eq:gA}
        &g_A
        &&= f_A^*(\restrict{h}{A})
        &&= e^{2 (\psi\circ f_A + \log |f_A'|)}
        &&\restrict{\flatmetric}{\annulus_\tau}
        &&= e^{2 \log |f_A'|}
        \, \restrict{\flatmetric}{\annulus_\tau},
        \\
        \label{eq:g2}
        &g_2
        &&= f_2^*(\restrict{h}{D_2})
        &&= e^{2 (\psi\circ f_2 + \log |f_2'|)} 
        &&\restrict{\flatmetric}{\disk^*}
        &&= e^{2 (\psi\circ f_2 - \psi + \log |f_2'|)}
        \, \restrict{h}{\disk^*}.
    \end{alignat}
    The three metrics altogether form a metric on $\hat \C$ that, however, may be discontinuous across $\gamma_1$ and $\gamma_2$.
    Using the expression for the Loewner potential in terms of zeta-regularized determinants of the Laplacian from Proposition~\ref{prop:lpot_detz}, we find by diffeomorphism invariance of the determinants that,
    \begin{equation}
    \begin{aligned}
        \lpot{\hat \C, 2}(\gamma_1, \gamma_2)
        -
        \lpot{\hat \C, 2}(e^{-2\pi \tau} S^1, S^1)
        &=
        \log \frac{
            \detz{\restrict{h}{e^{-2\pi \tau}\disk}}
            \detz{\restrict{h}{\annulus_\tau}}
            \detz{\restrict{h}{\disk^*}}
        }{
            \detz{\restrict{h}{D_1}}
            \detz{\restrict{h}{A}}
            \detz{\restrict{h}{D_2}}
        }
        \\ &=
        \log \frac{
            \detz{\restrict{\flatmetric}{e^{-2\pi \tau}\disk}}
            \detz{\restrict{\flatmetric}{\annulus_\tau}}
            \detz{\restrict{h}{\disk^*}}
        }{
            \detz{\restrict{g_1}{e^{-2\pi \tau}\disk}}
            \detz{\restrict{g_A}{\annulus_\tau}}
            \detz{\restrict{g_2}{\disk^*}}
        }
        .
    \end{aligned}
    \end{equation}
    Thus, we anticipate to obtain Equation~\eqref{eq:lpot_liouville_thm} by the application of the Polyakov--Alvarez anomaly formula~\eqref{eq:polyakov_alvarez} in Appendix~\ref{appendix:zeta} to the conformal factor in Equations~\eqref{eq:g1},~\eqref{eq:gA}, and~\eqref{eq:g2}.
    The full expression is
    \begin{align}
        & \notag
        \lpot{\hat \C, 2}(\gamma_1, \gamma_2)
        -
        \lpot{\hat \C, 2}(e^{-2\pi \tau} S^1, S^1)
        \\ =\; & \notag
        -\paformula\big(\log |f_1'|, \restrict{\flatmetric}{e^{2\pi \tau} \disk}\big)
        - 
        \paformula\big(\log |f_A'|, \restrict{\flatmetric}{e^{2\pi \tau} \disk}\big)
        - 
        \paformula\big(\log |f_2'|, \restrict{h}{e^{2\pi \tau} \disk}\big)
        \\ =\; &
        \label{eq:dirichlet_energy_terms}
        \frac{1}{6 \pi} \bigg(
            \iint_{e^{-2\pi \tau} \disk}
            \frac{1}{2} \big|\nabla \log |f_1'| \big|^2 
            \; \dz
            + \iint_{\annulus_\tau}
            \frac{1}{2} \big|\nabla \log |f_A'| \big|^2 
            \; \dz
        \\ &
        \label{eq:dirichlet_energy_term2}
        \phantom{\frac{1}{6 \pi} \bigg(}
            + \iint_{\disk^*}
            \frac{1}{2} \big|\nabla_h (\log |f_2'| + \psi \circ f_2 - \psi) \big|^2 
            \; \dvol{h}
        \\ &
        \label{eq:curvature_term}
        \phantom{\frac{1}{6 \pi} \bigg(}
            + \iint_{\disk^*}
            R_h \;
            (\log |f_2'| + \psi \circ f_2 - \psi)
            \; \dvol{h}
        \\ &
        \label{eq:N1}
        \phantom{\frac{1}{6 \pi} \bigg(}
            + \int_{e^{-2\pi \tau} S^1} \Big(
                e^{2 \pi \tau} (\log |f_1'| - \log |f_A'|) + \frac{3}{2} N (\log |f_1'| - \log |f_A'|)
            \Big) \; \dboundary{}
        \\ &
        \label{eq:N2}
        \phantom{\frac{1}{6 \pi} \bigg(}
            + \int_{S^1} \Big(
                (\log |f_A'| - \log |f_2'|) + \frac{3}{2} N (\log |f_A'| - \log |f_2'|)
            \Big) \; \dboundary{}
        \bigg)
    \end{align}
        
    Generally, the pre-Schwarzians appear because for a holomorphic function $f$,
    \begin{equation}
        \big|\nabla \log |f'|\big|^2 = \left|
            \preschwarzian{f}
        \right|^2
        .
    \end{equation}
    For the first term~\eqref{eq:dirichlet_energy_terms} involving $f_1$ and $f_A$, this is immediate, since the underlying metric is indeed the flat one.
    In the integrals over $\disk^*$, we have to deal with the non-flat metric~$h$.
    The Dirichlet energy in the term~\eqref{eq:dirichlet_energy_term2} is conformally invariant and the Gaussian curvature in the term~\eqref{eq:curvature_term} is $R_h = e^{- 2 \psi} \Delta \psi$ and the volume form is $\dvol{h} = e^{2 \psi} \dz$.
    Since these conformal factors cancel out, the terms~\eqref{eq:dirichlet_energy_term2} and~\eqref{eq:curvature_term} combine to
    \begin{align}
    \notag
        &
        \frac{1}{6\pi}
            \iint_{\disk^*} \bigg(
            \frac{1}{2} \big|\nabla (\log |f_2'| + \psi \circ f_2 - \psi) \big|^2 
            + (\Delta \psi) \;
            (\log |f_2'| + \psi \circ f_2 - \psi)
            \bigg)
            \; \dz
        \\ =\; &
        \label{eq:bulk_terms}
        \frac{1}{6\pi}
            \iint_{\disk^*} \bigg(
            \frac{1}{2} \big|\nabla (\log |f_2'| + \psi \circ f_2 - \psi) \big|^2 
            + \Big\langle
                \nabla \psi,
                \nabla (\log |f_2'| + \psi \circ f_2 - \psi)
            \Big\rangle
            \bigg)
            \; \dz
            ,
    \end{align}
    where the boundary term in the application of Green's first identity vanishes since $\psi = 0$ in a neighborhood of $S^1$.
    By reparametrization and then by conformal invariance, we have
    \begin{equation}
        \iint_{\disk^*}
        \frac{1}{2} \big|\nabla \psi \big|^2 
        \; \dz
        =
        \iint_{D_2}
        \frac{1}{2} \big|\nabla_{g_2} (\psi \circ f_2) \big|_{g_2}^2 
        \; \dvol{g_2}
        =
        \iint_{D_2}
        \frac{1}{2} \big|\nabla (\psi \circ f_2) \big|^2 
        \; \dz
    \end{equation}
    and we can change the domain of integration back to $\disk^*$ since the support of $\psi$ is contained in $\disk^* \cap D_2$.
    Thus, by expanding the expression~\eqref{eq:bulk_terms}, we find that most of the terms cancel out except for
    \begin{equation}
    \text{\eqref{eq:bulk_terms}} = 
        \frac{1}{6\pi}
        \iint_{\disk^*}
        \frac{1}{2} \big|\nabla \log |f_2'| \big|^2 
        \; \dz
        =
        \frac{1}{12\pi}
        \iint_{\disk^*}
        \big|\preschwarzian{f_2'} \big|^2 
        \; \dz
        .
    \end{equation}
    
    For the terms~\eqref{eq:N1} and~\eqref{eq:N2} involving the normal derivative $N$, observe that it is related to a conformal change $e^{2\sigma} h$ of the geodesic curvature by $k_g = e^{-\sigma} (k_h + N_h \sigma)$ (see, {e.g.\ \cite[Appendix~A]{wang_equivalent}}).
    Under the same conformal change, $\dboundary{g} = e^{\sigma} \dboundary{h}$. Finally, we find reparametrizations of $\int_{\gamma_1} k_h \dboundary{h}$ which cancel each other out,
    \begin{equation}
    \begin{aligned}
        &
        \int_{e^{-2\pi \tau} S^1}
            N (\log |f_1'| - \log |f_A'|)
        \dboundary{}
        \\ \; =&
        \int_{e^{-2\pi \tau} S^1} \bigg(
            (e^{\log |f_1'|} k_{g_1} - k_h)
            -
            (e^{\log |f_A'|} k_{g_A} - k_h)
        \bigg) \dboundary{h}
        \\ \; =&
        \int_{e^{-2\pi \tau} S^1}
            k_{g_1}\; e^{\log |f_1'|}\dboundary{h}
        -
        \int_{e^{-2\pi \tau} S^1}
            k_{g_A}\; e^{\log |f_A'|}\dboundary{h}
        = 0,
    \end{aligned}
    \end{equation}    
    and analogously for the integral over $S^1$.
    
    By the mean-value property of harmonic functions, we have
    \begin{equation}
        \int_{e^{-2\pi \tau} S^1}
        e^{2\pi \tau} \log |f_1'| \; \dboundary{}
        - \int_{S^1}
        \log |f_2'| \; \dboundary{}
        =
        - 2 \pi \log \left|
            \frac{f_2'(\infty)}{f_1'(0)}
        \right|
        .
    \end{equation}
    Finally, the remaining boundary integrals involving $\log|f_A'| = \Re(\log f_A')$ cancel out as the real part of the contour integrals
    \begin{equation}
        \int_{S^1} \log f_A' \;\dd z
        - \int_{e^{-2\pi \tau} S^1} \log f_A' \;\dd z
        = 0.
    \end{equation}
\end{proof}

\subsection{Loewner--Kufarev energy and a multiple Grunsky inequality}
\label{section:foliation_energy}

In this section, we find that the two-loop Loewner potential~\eqref{eq:lpot_def} of smooth loops $\gamma_1$,~$\gamma_2$ is related to the Loewner--Kufarev energy of a foliation obtained from the setup in Figure~\ref{fig:setup} by taking equipotential curves --- pushforwards of the foliation by circles along the respective uniformizing maps; see also Figure~\ref{fig:foliation}.
Our results are similar to those in \cite{viklund_wang_lk_energy} for the one-loop Loewner energy.
However, our setup requires a generalization of the Grunsky inequality, which we call a multiple Grunsky inequality because compared to usual Grunsky inequality it considers additional nested annuli between the disjoint simply connected domains.
Even though the proof is given only for the case of a single annulus, it has a direct generalization to any finite number of annuli --- see Remark~\ref{remark:more_annuli}.

Consider the generalized Grunsky inequality in \cite{Hu72} and also \cite[Theorem~4.1]{Pom75} of a pair Riemann mappings $f_1 : \disk \to D_1$ and $f_2 : \disk^* \to D_2$ such that $f_1(0) = 0$,
$f_2(\infty) = \infty$,
$f_2'(\infty) = 1$,
and $D_1 \cap D_2 = \emptyset$ and their Grunsky coefficients $b_{k, l}$ for $k, l \in \Z$ defined by the two-variable expansion of difference quotients
\begin{align}
\label{eq:grunsky2}
    \log \frac{f_2(z) - f_2(w)}{z - w}
    &= - \sum_{k=1}^\infty \sum_{l=1}^\infty
        b_{k,l} z^{-k}w^{-l},
        &&z, w \in \disk^*,
    \\
\label{eq:grunsky12}
    \log \frac{f_2(w) - f_1(z)}{w}
    &= - \sum_{k=0}^\infty \sum_{l=1}^\infty
        b_{-k, l} z^{k} w^{-l},
        &&z \in \disk^*, w \in \disk,
    \\
\label{eq:grunsky1}
    \log \frac{f_1(z) - f_1(w)}{z-w}
    &= - \sum_{k=0}^\infty \sum_{l=0}^\infty
        b_{-k, -l} z^{k} w^{l},
        &&z, w \in \disk
\end{align}
and $b_{k, l} = b_{l, k}$.
For any constants $\lambda_{-m}, \dots, \lambda_m \in \C$, the inequality reads,
\begin{equation}
\label{eq:grunsky_inequality_coefficients}
    \sum_{k \in \Z} |k| \; \bigg| \sum_{|l| \leq m} b_{k, l} \lambda_l \bigg|^2
    \leq
    \sum_{0 < |k| \leq m} \frac{|\lambda_k|^2}{k} + 2 \Re\bigg(\bar \lambda_0 \sum_{|l| \leq m} b_{0, l} \lambda_l \bigg),
\end{equation}
where equality holds if and only if $\hat \C \setminus (D_1 \cup D_2)$ is of measure zero.

The multiple Grunsky inequality proven below resembles the special case of \eqref{eq:grunsky_inequality_coefficients} in
\cite[Chapter~2, Remark~2.2]{takhtajan_teo}
with just one nonzero constant $\lambda_0 = 1$, that is, $m = 0$,
\begin{equation}
    \sum_{k \in \Z} |k| \; |b_{k, 0}|^2 \leq 2 \Re b_{0, 0}
    =
    2 \log \bigg|
        \frac{f_2'(\infty)}{f_1'(0)}
    \bigg|
    ,
\end{equation}
By taking $z = 0$ in Equation~\eqref{eq:grunsky12} and $w = 0$ Equation~\eqref{eq:grunsky1} and direct computation,
\begin{equation}
\label{eq:series_bound}
\begin{aligned}
    \iint\limits_{\disk^*}
    \bigg| 
        \frac{f_2'(z)}{f_2(z)} - \frac{1}{z}
    \bigg|^2
    \dz
    &=
    \iint\limits_{\disk^*}
    \bigg| 
        \bigg(
            \log \frac{f_2(z)}{z}
        \bigg)'
    \bigg|^2
    \dz
    = \pi \sum_{k = 1}^\infty k |b_{k , 0}|^2
    ,
    \\
    \iint\limits_{\disk}
    \bigg| 
        \frac{f_1'(z)}{f_1(z)} - \frac{1}{z}
    \bigg|^2
    \dz
    &=
    \iint\limits_{\disk}
    \bigg| 
        \bigg(
            \log \frac{f_1(z)}{z}
        \bigg)'
    \bigg|^2
    \dz
    = \pi \sum_{k = 0}^\infty k |b_{-k , 0}|^2
    .
\end{aligned}
\end{equation}
The special case of~\eqref{eq:grunsky_inequality_coefficients} becomes
\begin{equation}
    \iint\limits_{\disk}
    \bigg| 
        \frac{f_1'(z)}{f_1(z)} - \frac{1}{z}
    \bigg|^2
    \dz
    +
    \iint\limits_{\disk}
    \bigg| 
        \frac{f_2'(z)}{f_2(z)} - \frac{1}{z}
    \bigg|^2
    \dz
    \leq
    2\pi \log \bigg|
        \frac{f_2'(\infty)}{f_1'(0)}
    \bigg|
    .
\end{equation}
For the multiple Grunsky equality, replace $f_1 : \disk \to D_1$ by $f_1 : e^{2\pi \tau} \disk \to D_1$, $\tau > 0$ and introduce a third univalent function $f_A : \annulus_\tau \to \hat \C \setminus (D_1 \cup D_2)$. The setup is as in Figure~\ref{fig:setup} except that the complement of $D_1$, $D_2$, and $A = f_A(\annulus_\tau)$ may not be of measure zero --- see Figure~\ref{fig:grunsky_proof}.
The proof follows that of the generalized Grunsky inequality~\eqref{eq:grunsky_inequality_coefficients} in \cite[Theorem~4.1]{Pom75} by adding additional contour integrals and series expansions.
\begin{proposition}[Multiple Grunsky inequality]
    \label{prop:multiple_grunsky}
    Consider univalent functions
    \begin{equation}
        f_1 : e^{-2\pi \tau} \disk \to D_1, \qquad
        f_A : \annulus_\tau \to A, \qquad
        f_2 : \disk^* \to D_2,
    \end{equation}
    such that $D_1$, $A$ and $D_2$ are disjoint and $f_1(0) = 0$, $f_2(\infty) = \infty$. Then,
    \begin{equation}
        \iint\limits_{e^{-2\pi\tau} \disk}
        \bigg| 
            \frac{f_1'(z)}{f_1(z)} - \frac{1}{z}
        \bigg|^2
        \dz
        +
        \iint\limits_{\annulus_\tau}
        \bigg| 
            \frac{f_A'(z)}{f_A(z)} - \frac{1}{z}
        \bigg|^2
        \dz
        +
        \iint\limits_{\disk}
        \bigg| 
            \frac{f_2'(z)}{f_2(z)} - \frac{1}{z}
        \bigg|^2
        \dz
        \leq
        2\pi \log \bigg|
            \frac{f_2'(\infty)}{f_1'(0)}
        \bigg|
        .
    \end{equation}
    Equality holds if and only if $\hat \C \setminus (D_1 \cup A \cup D_2)$ is of measure zero.
\end{proposition}
\begin{figure}
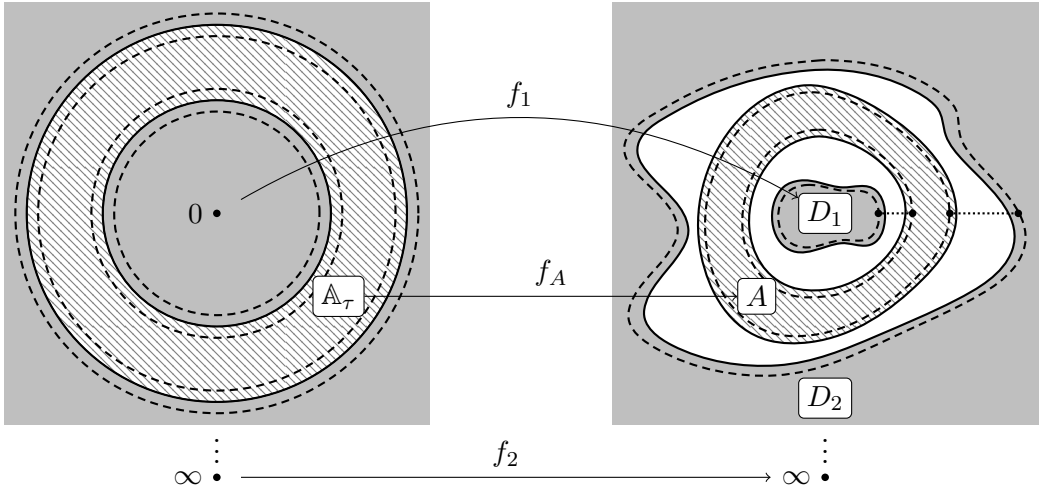

\centering
\includestandalone[]{figures/fig_grunsky}
\caption{
    The setup for the multiple Grunsky inequality, Proposition~\ref{prop:multiple_grunsky}.
    In contrast to Figure~\ref{fig:setup}, the complement of $D_1 \cup A \cup D_2$ may not be of measure $0$.
    The dashed lines on the left are circles $\varepsilon > 0$ away from $e^{-2\pi \tau}S^1$ and $S^1$ respectively. Their image --- the dashed lines on the right --- as well as the dotted lines, are part of the integration contour in the proof, see Equation~\eqref{eq:grunsky_contour}.
}
\label{fig:grunsky_proof}
\end{figure}
\begin{proof}
    The inequality will be obtained from the positivity of the Dirichlet energy of the function $\log z$ integrated in the complement ${F = \hat \C \setminus (D_1 \cup A \cup D_2)}$.
    We momentarily enlarge $F$ for some small $\varepsilon > 0$ to the interior of the curves, see Figure~\ref{fig:grunsky_proof},
    \begin{equation}
    \label{eq:grunsky_contour}
    \begin{aligned}
        C_1^\varepsilon &=
        - f_1\big((e^{-2\pi \tau} - \varepsilon)\; S^1\big)
        + T_1^+
        + f_A\big((e^{-2\pi \tau} + \varepsilon)\; S^1\big)
        - T_1^-,
        \\
        C_2^\varepsilon &=
        - f_A\big((1 - \varepsilon)\; S^1\big)
        + T_2^+
        + f_2\big((1 + \varepsilon)\; S^1\big)
        - T_2^-
        ,
    \end{aligned}
    \end{equation}
    where the signs stand for the reversal of the conventional orientation of the curves and $T_1^{\pm}$, $T_2^{\pm}$ are the left and right limits of segments connecting $f_1(e^{-2\pi \tau} - \varepsilon)$ to $f_A(e^{-2\pi\tau} + \varepsilon)$ and $f_A(1 - \varepsilon)$ to $f_2(1 + \varepsilon)$ respectively (The dotted lines in Figure~\ref{fig:grunsky_proof}).

    By the analytic Green's formula, see \cite[Theorem~2.2]{Pom75}, we have
    \begin{equation}
    \label{eq:dirichlet_grunsky}
        0 \leq \iint_{F} \frac{\dz}{|z|^2}
        = 
        \lim_{\varepsilon \to 0}
        \frac{1}{2 \ii} \int_{C_1^\varepsilon \cup C_2^\varepsilon} \frac{\log \bar z}{z} \; \dd z.
    \end{equation}
    Note that the left-hand side is independent of the choice of branch cut for $\log \bar z$, thus we can choose it such that it goes from $0$ to $\infty$ along $T_1^{\pm}$ and $T_2^{\pm}$ and does not intersect the curves elsewhere.
    The jump across the branch cut is $2\pi \ii$, thus
    \begin{equation}
    \label{eq:logjump}
        \frac{1}{2 \ii}
        \int_{{T_1^+ - T_1^- \cup T_2^+ - T_2^-}} \frac{\log \bar z}{z} \; \dd z
        =
        - \pi \int_{T_1^+ \cup T_2^+} \frac{\dd z}{z}
        =
        \pi
        \log \frac{
            f_A(e^{-2\pi \tau} + \varepsilon)
            f_2(1 + \varepsilon)
        }{
            f_1(e^{-2\pi \tau} - \varepsilon)
            f_A(1 - \varepsilon)
        }
        .
    \end{equation}
    Consider the expansions
    \begin{equation}
    \begin{aligned}
        \log f_1(z) &= \log f_1'(0) + \log z - \sum_{k = 1}^\infty b_{-k, 0} z^{k},
        &&z \in e^{-2\pi \tau} \disk,
        \\
        \log f_A(z) &= \log z - \sum_{k \in \Z} \beta_k z^{k},
        &&z \in \annulus_\tau,
        \\
        \log f_2(z) &= \log f_2'(\infty) + \log z - \sum_{k = 1}^\infty b_{k, 0} z^{-k},
        &&z \in \disk^*
        .
    \end{aligned}
    \end{equation}
    On the one hand, similar to Equation~\eqref{eq:series_bound}, we have
    \begin{equation}
    \label{eq:as_series}
    \begin{gathered}
        \iint\limits_{e^{-2\pi\tau} \disk}
        \bigg| 
            \frac{f_1'(z)}{f_1(z)} - \frac{1}{z}
        \bigg|^2
        \dz
        +
        \iint\limits_{\annulus_\tau}
        \bigg| 
            \frac{f_A'(z)}{f_A(z)} - \frac{1}{z}
        \bigg|^2
        \dz
        \iint\limits_{\disk^*}
        \bigg| 
            \frac{f_2'(z)}{f_2(z)} - \frac{1}{z}
        \bigg|^2
        \dz
        \\ =
        \pi \sum_{k = 1}^\infty k \; |b_{-k , 0}|^2 e^{-4 \pi k \tau}
        + \pi \sum_{k \in \Z} k |\beta_{k}|^2 (1 - e^{-4 \pi k \tau})
        + \pi \sum_{k = 1}^\infty k \; |b_{k , 0}|^2
        .
    \end{gathered}
    \end{equation}
    On the other hand, we evaluate the loops in the contour integral~\eqref{eq:grunsky_contour} as $\varepsilon \to 0$.
    First compute,
    \begin{equation}
    \begin{aligned}
        &\overline{\log f_1(z)} \; (\log f_1(z))' \dd z \big|_{z = e^{-2\pi \tau} e^{\ii t}}
        \\ = \; &
        \ii \Big(
            \overline{\log f'(0)} - 2 \pi \tau - \sum_{k = 1}^\infty \bar b_{-k, 0} e^{-2\pi \tau} e^{-\ii k t}  - \ii t
        \Big) \Big(
            1 - \sum_{k = 1}^\infty k b_{-k, 0} e^{-2\pi k \tau} e^{\ii k t}
        \Big)
        \dd t
    \end{aligned}
    \end{equation}
    and likewise for $f_A$ and $f_2$.
    After some computation, the respective integrals over loops in Equation~\eqref{eq:dirichlet_grunsky} become
    \begin{equation}
    \label{eq:grunsky_integral_loops}
    \begin{aligned}
        - \frac{1}{2\ii} \int_{f_1(e^{-2\pi \tau}\; S^1)} \frac{\log \bar z}{z} \; \dd z
        &=
        - \pi \Big(
        \overline{\log f_1'(0)}
        + \sum_{k = 1}^\infty k |b_{-k, 0}|^2 e^{-4 \pi k \tau}
        - 4 \pi \tau
        - 2 \ii
        \\ &\phantom{=
        - \pi \Big(}
        - \log f_1(e^{-2\pi \tau})
        + \log f_1'(0)
        \Big)
        \\
        \frac{1}{2\ii} \int_{f_A(e^{-2\pi \tau}\; S^1)} \frac{\log \bar z}{z} \; \dd z
        &=
        \pi \Big(
        \sum_{k \in \Z} k |\beta_k|^2 e^{-4 \pi k \tau}
        - 4 \pi \tau
        - 2 \ii
        - \log f_A(e^{-2 \pi \tau})
        \Big)
        \\
        - \frac{1}{2\ii} \int_{f_A(S^1)} \frac{\log \bar z}{z} \; \dd z
        &=
        - \pi \Big(
        \sum_{k \in \Z} k |\beta_k|^2
        - 2 \ii
        - \log f_A(1)
        \Big)
        \\
        \frac{1}{2\ii} \int_{f_2(S^1)} \frac{\log \bar z}{z} \; \dd z
        &=
        \pi \Big(
        \overline{\log f_2'(\infty)}
        - \sum_{k = 1}^\infty k |b_{k, 0}|^2
        - 2 \ii
        - \log f_2(1)
        + \log f_2'(\infty)
        \Big)
    \end{aligned}
    \end{equation}
    The right-hand side of the inequality~\eqref{eq:dirichlet_grunsky} equals the sum of Equations~\eqref{eq:grunsky_integral_loops} and Equation~\eqref{eq:logjump} with $\varepsilon = 0$. Since several terms cancel out, we find
    \begin{equation}
        0 \leq
        - \pi \sum_{k = 1}^\infty k |b_{-k, 0}|^2 e^{-4 \pi k \tau}
        - \pi \sum_{k \in \Z} k |\beta_{k}|^2 (1 - e^{-4 \pi k \tau})
        - \pi \sum_{k = 1}^\infty k |b_{k, 0}|^2
        + 2\pi \log \bigg|
            \frac{f_2'(\infty)}{f_1'(0)}
        \bigg|
        ,
    \end{equation}
    from which the result follows by comparing to Equation~\eqref{eq:as_series}.
    
    Moreover, equality holds if the Dirichlet energy~\eqref{eq:dirichlet_grunsky} vanishes, which is the case if and only if $H$ is of measure zero.
\end{proof}

\begin{remark}
    \label{remark:more_annuli}
    Proposition~\ref{prop:multiple_grunsky} can be further generalized to the case of any finite number of nested annuli replacing $f_A$. The proof is analogous by introducing further contours.
\end{remark}

We now turn to the relation of the two-loop Loewner potential~\eqref{eq:lpot_def} and the Loewner--Kufarev energy of the foliation $(\foliation_t)_{t \in \R}$ defined by
\begin{equation}
    \foliation_t= \begin{cases}
        f_2(e^{-2\pi t} S^1), & t \leq 0, \\
        f_A(e^{-2\pi t} S^1), & 0 < t < \tau, \\
        f_1(e^{-2\pi t} S^1), & t \geq \tau,
    \end{cases}
\end{equation}
where $f_1$, $f_A$, $f_2$ are the uniformizing maps of the complement of a pair of non-intersecting simple smooth loops $\gamma_1$, $\gamma_2$ like in Figure~\ref{fig:setup}.
This is the foliation by equipotential loops depicted in Figure~\ref{fig:foliation}.

\begin{figure}
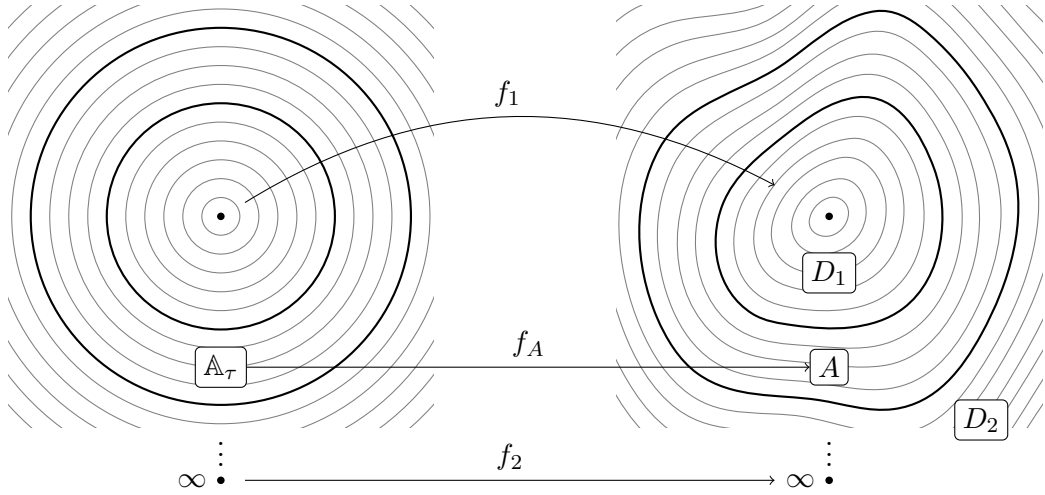

\centering
\includestandalone[]{figures/fig_foliation}
\caption{
    The pushforward of the foliation of $\C \setminus \{0\}$ by concentric circles by the conformal maps $f_1$, $f_A$, $f_2$ as in Figure~\ref{fig:setup} which uniformize the complements of the highlighted loops. 
}
\label{fig:foliation}
\end{figure}

By the energy duality theorem \cite[Theorem~1.2]{viklund_wang_lk_energy}, the Loewner--Kufarev energy of the specific foliation $(\foliation_t)_{t \in \R}$ may be defined as
\begin{equation}
    \label{eq:lkenergy_def}
    \lkenergy((\foliation_t)_{t \in \R}) = \frac{1}{16\pi}
    \iint_{\C} |\nabla \varphi(z)|^2 \dz, \qquad
    \varphi(z) = \begin{cases}
        \arg \frac{z \; f_1'(z)}{f_1(z)}, & z \in e^{-2\pi \tau} \disk, \\
        \arg \frac{z \; f_A'(z)}{f_A(z)}, & z \in \annulus_\tau, \\
        \arg \frac{z \; f_2'(z)}{f_2(z)}, & z \in \disk^*.
    \end{cases}
\end{equation}
This the Dirichlet energy of the winding function $\varphi$ defined above, where the argument is defined as the continuous branch going to $0$ as $z \to \infty$.
\begin{theorem}
    \label{thm:lkenergy}
    For two non-intersecting simple smooth loops $\gamma_1$, $\gamma_2$ in $\hat \C$ with the domains and conformal maps set up as in Figure~\ref{fig:setup}, the two-loop Loewner potential~\eqref{eq:lpot_def} equals
    \begin{equation}
            \lpot{\hat \C, 2}(\gamma_1, \gamma_2)
            =
            \lpot{\hat \C, 2}(e^{-2\pi\tau} S^1, S^1)
            + \frac{4}{3}\lkenergy((\foliation_t)_{t \in \R})
            - \frac{1}{6} \log \left|
                \frac{
                    f_2'(\infty)
                }{
                    f_1'(0)
                }
            \right|
            + \const
            .
    \end{equation}
    with the winding function $\varphi$ and Loewner--Kufarev energy $\lkenergy$ as in Equation~\eqref{eq:lkenergy_def}.
\end{theorem}
\begin{proof}
    For a generic holomorphic univalent function we have,
    \begin{equation}
        \bigg|\nabla \arg \frac{z \; f'(z)}{f(z)}\bigg|^2
        = \bigg|\preschwarzian{f} - \frac{f'}{f} + \frac{1}{z} \bigg|^2
        = \big|\preschwarzian{f}\big|^2
        + \bigg|\frac{f'}{f} - \frac{1}{z}\bigg|^2
        - 2 \Re \bigg(\preschwarzian{f} \overline{\bigg(\frac{f'}{f} - \frac{1}{z}\bigg)}\bigg)
        .
    \end{equation}
    To find similar terms from the pre-Schwarzian in Theorem~\ref{thm:lpot_liouville}, we study the inverted loops $\inversion(\gamma_1)$ and $\inversion(\gamma_2)$ where $\inversion(z) = \frac{1}{z}$. The domains complementary to the inverted pair are uniformized by the functions
    \begin{align}
        \inversion \circ f_2 \circ \inversion \circ (e^{2\pi\tau} \blank) &: e^{2 \pi \tau} \disk \to \inversion(D_2), \\
        \inversion \circ f_A \circ \inversion \circ (e^{2\pi\tau} \blank) &: \annulus_\tau \to \inversion(A), \\
        \inversion \circ f_1 \circ \inversion \circ (e^{2\pi\tau} \blank) &: \disk \to \inversion(D_1).
    \end{align}
    Generically, the pre-Schwarzian of such a map is
    \begin{equation}
    \begin{gathered}
        \big|\preschwarzian{\inversion \circ f \circ \inversion \circ (e^{2\pi\tau} \blank)}(z)\big|^2
        =
        \frac{e^{-4 \pi \tau}}{|z|^4}
        \bigg|
            \frac{2 f'\big(\frac{1}{e^{2 \pi \tau}z}\big)}{f\big(\frac{1}{e^{2 \pi \tau}z}\big)}
            - \frac{f''\big(\frac{1}{e^{2 \pi \tau}z}\big)}{f'\big(\frac{1}{e^{2 \pi \tau}z}\big)}
            - 2 e^{2 \pi \tau} z
        \bigg|^2
        = \\
        \frac{e^{-4 \pi \tau}}{|z|^4} \Bigg(
        \big|\preschwarzian{f}({\textstyle \frac{1}{e^{2 \pi \tau}z}})\big|^2
        + 4 \bigg|
            \frac{f'(\frac{1}{e^{2 \pi \tau}z})}{f(\frac{1}{e^{2 \pi \tau}z})}
            - e^{2 \pi \tau}z
        \bigg|^2
        - 4 \Re \bigg(
            \preschwarzian{f}({\textstyle \frac{1}{e^{2 \pi \tau}z}})
            \overline{\bigg(
                \frac{f'(\frac{1}{e^{2 \pi \tau}z})}{f(\frac{1}{e^{2 \pi \tau}z})}
                - e^{2 \pi \tau}z
            \bigg)}
        \bigg)
        \Bigg)
        .
    \end{gathered}
    \end{equation}
    Therefore, we have the following formula under the coordinate transformation $w = \frac{1}{e^{2 \pi \tau}z}$,
    \begin{equation}
        \big|\preschwarzian{\inversion \circ f \circ \inversion \circ (e^{2 \pi \tau} \blank)}(w)\big|^2 \dw
        =
        \bigg(
        - |\preschwarzian{f}(z)|^2
        + 2 \; \bigg|\frac{f'}{f} - \frac{1}{z}\bigg|^2
        + 
        2 \; \bigg|\nabla \arg \frac{z \; f'(z)}{f(z)}\bigg|^2
        \bigg)
        \dz
        .
    \end{equation}
    A direct computation shows that
    \begin{equation}
    \label{eq:coordinate_transform_inversion}
        \log \bigg|
            \frac{
                (\inversion \circ f_2 \circ \inversion \circ (e^{2 \pi \tau} \blank))'(0)
            }{
                (\inversion \circ f_1 \circ \inversion \circ (e^{2 \pi \tau} \blank))'(\infty)
            }
        \bigg|
        =
        - \log \bigg|
            \frac{f_2'(\infty)}{f_1'(0)}
        \bigg|
        .
    \end{equation}
    Note how the pre-Schwarzians in Equation~\eqref{eq:coordinate_transform_inversion} cancel with those in Theorem~\ref{thm:lpot_liouville} when applied to the non-inverted loops.
    The other terms to the left of Equation~\eqref{eq:coordinate_transform_inversion} may be rewritten by application of the Grunsky inequality, Proposition~\ref{prop:multiple_grunsky}, which is an equality here, and the definition~\eqref{eq:lkenergy_def} of Loewner--Kufarev energy.
    By combining these findings, we obtain
    \begin{equation}
    \begin{aligned}
        &\lpot{\hat \C, 2}(\gamma_1, \gamma_2)
        =
        \frac{1}{2} \bigg(
            \lpot{\hat \C, 2}(\gamma_1, \gamma_2)
            +
            \lpot{\hat \C, 2}(\inversion(\gamma_2), \inversion(\gamma_2))
        \bigg)
        \\ = \; &
        \lpot{\hat \C, 2}(e^{-2\pi\tau} S^1, S^1)
        + \frac{4}{3}\lkenergy((\foliation_t)_{t \in \R})
        - \frac{1}{6}
        \log \bigg|
            \frac{f_2'(\infty)}{f_1'(0)}
        \bigg|
        + \const.
    \end{aligned}
    \end{equation}
\end{proof}

\section{Variation of the two-loop Loewner potential}
\label{section:variations_sle}

\subsection{Variations preserving the modulus}

In this section, we employ the following variational formula of the one-loop Loewner potential \cite{takhtajan_teo, wang_equivalent, sung_wang} to find variations of the two-loop Loewner potential preserving the modulus of the annulus between the loops:
\begin{equation}
    \label{eq:one_loop_variation}
    \eval{\pdv{\varepsilon}}_{\varepsilon = 0}
    \lpot{\hat \C}(\omega^{\varepsilon \nu}(\gamma))
    = - \frac{1}{3\pi} \left(
        \iint_{D_1} \nu \, \schwarzian{f_1^{-1}} \, \dz
        + \iint_{D_2} \nu \, \schwarzian{f_2^{-1}} \, \dz
    \right)
    .
\end{equation}
Here, $D_1$ and $D_2$ are the bounded and unbounded connected components of $\hat \C \setminus \gamma$ and $f_1 : \disk \to D_1$, $f_2 : \disk^* \to D_2$ are Riemann mappings. The direction of the variation is characterized by an infinitesimal Beltrami differential $\nu$ on $\hat \C$ with compact support away from the curve $\gamma$. Also, $\omega^{\varepsilon \nu} : \hat \C \to \hat \C$ is a quasiconformal solution of the Beltrami equation associated to $\varepsilon \nu$ for small $\varepsilon \geq 0$.

To make use of Equation~\eqref{eq:one_loop_variation}, we need a cascade relation for the two-loop Loewner potential --- analogous to Equation~\eqref{eq:cascade} for the $\SLE_\kappa$ two-loop measure.
The following is a direct consequence of Proposition~\ref{prop:lpot_detz}:
\begin{equation}
\label{eq:cascade_lpot}
\begin{aligned}
    \lpot{\hat \C, 2}(\gamma_1, \gamma_2)
    &=
    \lpot{\hat \C}(\gamma_1)
    + \log \frac{
    \detz{\restrict{g}{A \cup D_2}}
    }{
    \detz{\restrict{g}{A}}
    \detz{\restrict{g}{D_2}}
    }
    + \const
    \\ &=
    \lpot{\hat \C}(\gamma_2)
    + \log \frac{
    \detz{\restrict{g}{D_1 \cup A}}
    }{
    \detz{\restrict{g}{D_1}}
    \detz{\restrict{g}{A}}
    }
    + \const
    ,
\end{aligned}
\end{equation}
where the geometric setup described in Figure~\ref{fig:setup} is used.
\begin{proposition}
    \label{prop:two_loop_variation}
    Let $\nu$ be an infinitesimal Beltrami differential with compact support in the interior of $D_1 \cup D_2$ and let $f_1 : e^{-2 \pi \tau} \disk \to D_1$, $f_2 : \disk^* \to D_2$ be any Riemann mappings.
    Then,
    \begin{equation}
    \label{eq:two_loop_variation}
    \eval{\pdv{\varepsilon}}_{\varepsilon = 0}
        \lpot{\hat \C, 2}(\omega^{\varepsilon \nu}(\gamma_1), \omega^{\varepsilon \nu}(\gamma_2))
        = - \frac{1}{3\pi} \left(
            \iint_{D_1} \nu \, \schwarzian{f_1^{-1}} \, \dz
            + \iint_{D_2} \nu \, \schwarzian{f_2^{-1}} \, \dz
        \right).
    \end{equation}
    The variation is independent of the choice of Riemann mappings $f_1$, $f_2$.
\end{proposition}
\begin{proof}
    We use linearity of the variation by decomposing the infinitesimal Beltrami differential as $\nu = \nu_1 + \nu_2$, where $\nu_j$ has support in $D_j$, $j = 1, 2$.
    
    By the cascade relation \eqref{eq:cascade_lpot},
    \begin{equation}
        \label{eq:cascade_lpot_nu}
        \lpot{\hat \C}(\omega^{\varepsilon \nu_1}(\gamma_1), \omega^{\varepsilon \nu_1}(\gamma_2))
        =
        \lpot{\hat \C}(\omega^{\varepsilon \nu_1}(\gamma_1))
        + \log \frac{
        \detz{\restrict{g}{\omega^{\varepsilon \nu_1}(A \cup D_2)}}
        }{
        \detz{\restrict{g}{\omega^{\varepsilon \nu_1}(A)}}
        \detz{\restrict{g}{\omega^{\varepsilon \nu_1}(D_2)}}
        }
        + \const.
    \end{equation}
    Note that $\omega^{\varepsilon \mu_1}$ is holomorphic on $A \cup D_2$ and the ratio of determinants in \eqref{eq:cascade_lpot_nu} is independent of $g$. Thus, this term is independent of $\varepsilon$:
    \begin{equation}
        \frac{
        \detz{\restrict{g}{\omega^{\varepsilon \nu_1}(A \cup D_2)}}
        }{
        \detz{\restrict{g}{\omega^{\varepsilon \nu_1}(A)}}
        \detz{\restrict{g}{\omega^{\varepsilon \nu_1}(D_2)}}
        }
        =
        \frac{
        \detz{\restrict{\omega^{(\varepsilon \nu_1})^* g}{A \cup D_2}}
        }{
        \detz{\restrict{\omega^{(\varepsilon \nu_1})^* g}{A}}
        \detz{\restrict{\omega^{(\varepsilon \nu_1})^* g}{D_2}}
        }
        =
        \frac{
        \detz{\restrict{g}{A \cup D_2}}
        }{
        \detz{\restrict{g}{A}}
        \detz{\restrict{g}{D_2}}
        }
        .
    \end{equation}
    From the variation \eqref{eq:one_loop_variation} of the first term we find
    \begin{equation}
        \eval{\pdv{\varepsilon}}_{\varepsilon = 0}
        \lpot{\hat \C}(\omega^{\varepsilon \nu_1}(\gamma_1), \omega^{\varepsilon \nu_1}(\gamma_2))
        = - \frac{1}{3\pi}
            \iint_{D_1} \nu_1 \, \schwarzian{f_1^{-1}} \, \dz
        .
    \end{equation}
    The result \eqref{eq:two_loop_variation} is obtained by combining the above with the analogous formula for $\nu_2$.
\end{proof}

\subsection{Variation of the modulus}
\label{section:lpot_variation_modulus}

The variational formula of Proposition~\ref{prop:two_loop_variation} implies that at any critical point of $\lpot{\smash{\hat \C, 2}}(\gamma_1, \gamma_2)$ it must hold that $\gamma_1$ and $\gamma_2$ are circles in $\hat \C$ (such that $f_1$ and $f_2$ are Möbius transformations).
By conformal invariance, such a configuration is determined by the modulus $\tau$ of the annulus $A$ between the circles $\gamma_1$ and $\gamma_2$. However, the variations in Proposition~\ref{prop:two_loop_variation} preserve this modulus.
By applying explicit formulas for the zeta-regularized determinant of the Laplacian in the flat metric on a disk or on an annulus with circular boundary, we obtain the following formula.
\begin{proposition}
    \label{prop:lpot_circles}
    Let $\gamma_1$ and $\gamma_2$ be disjoint circles in $\hat \C$ and denote by $\tau > 0$ the modulus of the annulus enclosed by the circles. Then,
    \begin{equation}
        \lpot{\hat \C, 2}(\gamma_1, \gamma_2) =
        - \log \tau
        - 2 \log \phi(e^{-4 \pi \tau})
        + \const,
    \end{equation}
    where the constant is independent of $\tau$ and $\phi(x) = \prod_{k = 1}^\infty (1 - x^k)$ is the Euler function~\eqref{eq:euler}.
\end{proposition}
\begin{proof}
    By conformal invariance, without loss of generality, assume that $\gamma_1 = e^{-2\pi \tau} \, S^1$ and $\gamma_2 = S^1$.
    The annulus between the curves becomes
    \begin{equation}
        \annulus_\tau = \setsuchthatinline{z \in \C}{e^{-2\pi \tau} \leq |z| \leq 1}.
    \end{equation}
    Since the expression for $\lpot{\smash{\hat \C, 2}}$ in Proposition~\ref{prop:lpot_detz} is independent of the metric $g$ on $\hat \C$, we fix a metric $g$ which is the flat metric $\flatmetric$ on $\disk$ and extends smoothly to $\hat \C$ (this choice is independent of $\tau$).
    Using the expressions~\eqref{eq:det_disk} and~\eqref{eq:det_annulus}, we find that up to constants not depending on $\tau$,
    \begin{align}
        \lpot{\hat \C, 2}(e^{- 2\pi \tau}\, S^1, S^1)
        &=
        \log \frac{
            \detz{\restrict{g}{\hat \C}}
        }{
            \detz{\restrict{\flatmetric}{e^{-2\pi \tau} \disk}}
            \detz{\restrict{\flatmetric}{
                \annulus_\tau
            }}
            \detz{\restrict{g}{\disk^*}}
        }
        \nonumber
        \\ &=
        - \log
        \detz{\restrict{\flatmetric}{e^{-2\pi \tau} \disk}}
        - \log \detz{\restrict{\flatmetric}{
            \annulus_\tau
        }}
        + \const
        \nonumber
        \\ &=
        - \frac{2 \pi}{3} \tau
        + \frac{2 \pi}{3} \tau
        - \log(2 \pi \tau)
        - 2 \log \phi(
            e^{-4 \pi \tau}
        )
        + \const
        \nonumber
        \\ &=
        - \log(\tau)
        - 2 \log \phi(
            e^{-4 \pi \tau}
        )
        + \const
        \nonumber
        \\ &=
        - \log \tau
        - 2 \sum_{k = 1}^\infty \log(1 - e^{-4 \pi k \tau})
        + \const
        \label{eq:lpot_decreasing}
        .
    \end{align}
\end{proof}

Observe from Equation~\eqref{eq:lpot_decreasing}, that $\lpot{\smash{\hat \C, 2}}(e^{-2\pi \tau} S^1, S^1)$ is a decreasing function of $\tau$. The $\tau \to 0$ limit, in which the curves merge, is $\infty$ and the $\tau \to \infty$ limit, pushing the curves apart, is $-\infty$. Thus, the infimum of the two-loop Loewner potential is
\begin{equation}
    \inf_{(\gamma_1, \gamma_2)} \lpot{\hat \C, 2}(\gamma_1, \gamma_2) = - \infty
\end{equation}
and it is not attained. Therefore, it is not possible to define a two-loop Loewner energy by subtracting the minimum from the Loewner potential in the same way as for the one-loop Loewner potential and energy (see Equation~\eqref{eq:lenergy_one}).

\section{CFT partition functions and the two-loop Loewner potential}
\label{section:cft_partition_functions}

In this section, we discuss two results about the generalized notion of two-loop Loewner potential $\smash{\lpotx{\hat \C, 2}{Z}(\gamma_1, \gamma_2)}$ introduced in equation~\eqref{eq:def_lpot_z_two}.
First, we show that the definition in terms of zeta-regularized determinants of Laplacians in Equation~\eqref{eq:lpot_two_detz} from the generalization may be recovered from the generalization.
Since the definition of $\smash{\lpotx{\hat \C, 2}{Z}(\gamma_1, \gamma_2)}$ takes ideas from CFT, we conceptually find the partition functions associated to the usual probabilistic definition of one-loop SLE.
Secondly, we show results analogous to those in Section~\ref{section:variations_sle} for the probabilistic one-loop Loewner potential.
This results in the condition for the existence of minimizers of two-loop Loewner potential mentioned in Equation~\eqref{eq:criterion}, see Proposition~\ref{prop:lpot_z_circles}.
This section draws on the derivation of Equation~\eqref{eq:def_lpot_z_two} via the real determinant line bundle in Appendix~\ref{appendix:detrc}.

\subsection{The partition functions of loop SLE}
\label{section:trivialization_sle}

Malliavin--Kontsevich--Suhov (MKS) loop measures, following the original description by Kontsevich and Suhov \cite{kontsevich_suhov}, take values in a real determinant line bundle over the space of loops.
This means that an actual loop measure $\sleloopx{\hat \C}{Z}$ (taking values in $\Rp$) is obtained by picking a trivialization $Z$ of the line bundle.
In fact, as explained in \cite[Section~3.2]{maibach_peltola} and in Appendix~\ref{appendix:detrc}, such a trivialization $Z(\gamma, \hat \C)$ at a loop $\gamma \subset \hat \C$ is identified in a canonical way simply with a positive real number,
\begin{equation}
    \label{eq:triv_number}
    Z(\gamma, \hat \C) \mapsto e^{\frac{\charge}{2}\lpotx{\hat \C}{Z}(\gamma)} \in \Rp
    .
\end{equation}
In this section, we explain why the function $\lpotx{\hat \C}{Z}(\gamma)$ of the loop $\gamma$ is defined as the Loewner potential with respect to~$Z$ in Definition~\ref{def:loewner_potential_Z} in Appendix~\ref{appendix:detrc}.

Still following \cite{kontsevich_suhov}, in the comparison of trivializations $Z_1$ and $Z_2$, the measures $\sleloopx{\hat \C}{Z_1}$ and $\sleloopx{\hat \C}{Z_1}$ are absolutely continuous with Radon-Nikodym derivative being the ratio of the numbers \eqref{eq:triv_number},
\begin{equation}
    \frac{
        \dd \sleloopx{\hat \C}{Z_2}
    }{
        \dd \sleloopx{\hat \C}{Z_1}
    }(\gamma)
    = e^{\frac{\charge}{2} \big(
        \lpotx{\hat \C}{Z_2}(\gamma)
        -
        \lpotx{\hat \C}{Z_1}(\gamma)
    \big)}.
\end{equation}

Since the real determinant line bundle is a trivial line bundle, the choice of trivialization is usually omitted when discussing $\SLE_\kappa$ within the context of probability theory.
In this work, however, we are interested in the properties of two-loop Loewner potential $\lpotx{\hat \C, 2}{Z}$ considering various trivializations.
Specifically, we discuss $\SLE_\kappa$ within the context of CFT, where the trivialization of the real determinant line bundle is given by the partition functions of the CFT --- as suggested by the notation $Z$.

Details on the relation between the real determinant line bundle over surfaces and Loewner potentials are explained in Appendix~\ref{appendix:detrc}.
We just mention here that the restriction covariance property \eqref{eq:restriction_covariance} generalizes to the loop measure relative to a trivialization $Z$.
Namely, for $D \subset \hat \C$ we have
\begin{equation}
    \label{eq:restriction_covariance_z}
    \frac{
        \dd \sleloopx{D}{Z}
    }{
        \dd \sleloopx{\hat \C}{Z}
    }
    (\gamma)
    =
    \boldone_{\{\gamma \subset D\}} \;
    e^{\frac{\charge}{2}\big(
        \lpotx{D}{Z}(\gamma)
        -
        \lpotx{\hat \C}{Z}(\gamma)
    \big)}.
\end{equation}

We proceed to identify the trivialization of the real determinant line bundle associated to the usual $\SLE_\kappa$ loop measure $\sleloop{\hat \C}$.
Comparing the restriction covariance formulas \eqref{eq:restriction_covariance} and \eqref{eq:restriction_covariance_z}, we find that for $\sleloopx{\hat \C}{Z} = \sleloop{\hat \C}$ and $\sleloopx{D}{Z} = \sleloop{D}$ given a compact, simply connected subset $D \subset \C$ that
\begin{equation}
    \lpotx{D}{Z}(\gamma)
    -
    \lpotx{\hat \C}{Z}(\gamma)
    =
    \blmpairr(\gamma, \hat \C \setminus D)
    =
    \blmpairr(\gamma, \partial D)
    .
\end{equation}
By Equation~\eqref{eq:blm_detz} with $\gamma_1 = \gamma$ and $\gamma_2 = \partial D$, this holds for 
\begin{equation}
\begin{aligned}
    \lpotx{\hat \C}{Z}(\gamma)
    &=
    \log \frac{
        \detz{\restrict{g}{\hat \C}}
    }{
        \detz{\restrict{g}{D_1}}
        \detz{\restrict{g}{A \cup D_2}}
    }
    &&=
    \lpot{\hat \C}(\gamma),
    \\
    \lpotx{D}{Z}(\gamma)
    &=
    \log \frac{
        \detz{\restrict{g}{D}}
    }{
        \detz{\restrict{g}{D_1}}
        \detz{\restrict{g}{A}}
    }
    &&=
    \lpot{D}(\gamma)
\end{aligned}
\end{equation}
By \eqref{eq:def_lpot_z_one}, we conclude that the trivialization commonly used for SLE is the one of Example~\ref{example:trivializations}, Item~\ref{item:zeta_trivialization},
\begin{equation}
    \label{eq:triv_detz}
    Z(D) = \left(\frac{
        \detz{\restrict{g}{D}}
    }{
        \boundaryterm{g}(D)
    }\right)^{-\charge/2} [g] \in \Detrc(D).
\end{equation}
where $D \subseteq \hat \C$ is any subdomain.

\subsection{Variation of CFT Loewner potentials}
\label{section:variations_cft}

As explained in Appendix~\ref{appendix:detrc}, the real determinant line bundle over surfaces is defined over the corresponding moduli spaces of Riemann surfaces.
Since the moduli space containing $\hat \C$ or a simply connected subset $D \subset \hat \C$ only contains a single equivalence class, any two trivializations of $\Detrc$ over these moduli spaces are related by a multiplicative constant.
In particular, any trivialization $Z(D) = Z_g(D) [g] \in \Detrc(D)$ over any closed simply connected subset $D \subseteq \hat \C$ is related to the trivialization \eqref{eq:triv_detz} by a constant independent of~$D$,
\begin{equation}
    \label{eq:z_simply}
    \log Z_g(D) 
    =
    - \frac{\charge}{2} \log \frac{
        \detz{\restrict{g}{D}}
    }{
        \boundaryterm{g}(D)
    }
    + \const.
\end{equation}

Since the Loewner potential of a single loop $\gamma \subset \hat \C$ only involves $\hat \C$ and the simply connected domains $D_1$ and $D_2$, we find
\begin{equation}
    \label{eq:lpot_z_one}
    \lpotx{\hat \C}{Z}(\gamma) = \lpot{\hat \C}(\gamma) + \const,
\end{equation}
for a constant independent of $\gamma$.
Therefore, the one-loop Loewner energy may be expressed using any trivialization $Z$ without depending on the choice:
\begin{equation}
    \lenergy{\hat \C}(\gamma) = 12 \Big(\lpotx{\hat \C}{Z}(\gamma) - \inf_{\eta} \lpotx{\hat \C}{Z}(\gamma)\Big).
\end{equation}
This explains why in the study of MKS loop measures for a single loop in $\hat \C$, the possibility of having different trivializations for the real determinant line bundle does not play a role.
In particular, the same variational formula~\eqref{eq:one_loop_variation} holds for any $\lpotx{\hat \C}{Z}$.

For a two-loop configuration $(\gamma_1, \gamma_2)$ on $\hat \C$, the relation~\eqref{eq:lpot_z_one} is more complex.
The reason is that the two loops divide $\hat \C$ into two simply connected subsets $D_1$ and $D_2$ as well as a doubly-connected annulus $A$.
For the annulus, the analogue of Equation~\eqref{eq:z_simply} is
\begin{equation}
    \label{eq:z_annulus}
    \log Z_g(A) 
    =
    - \frac{\charge}{2}
    \log \frac{
        \detz{\restrict{g}{A}}
    }{
        \boundaryterm{g}(A)
    }
    + f_Z(\tau)
    ,
\end{equation}
for some function $f_Z(\tau) \in \R$
of the modulus $\tau$ of the annulus $A$. The function $f_Z(\tau)$ may be computed in the flat metric $\flatmetric$ on the standard annulus $\setsuchthat{z \in \C}{e^{-2\pi \tau} \leq |z| \leq 1}$,
\begin{equation}
    f_Z(\tau) = \log Z_{\flatmetric}(
        \annulus_\tau
        )
    + \frac{\charge}{2} \log \frac{
        \detz{\restrict{\flatmetric}{
            \annulus_\tau
        }}
    }{
        \boundaryterm{g}(
            \annulus_\tau
        )
    }.
\end{equation}
Comparing the two-loop Loewner potentials of the trivialization~\eqref{eq:triv_detz} and any other trivialization $\globaltriv$, we find
\begin{equation}
    \label{eq:lpot_z_two}
    \lpotx{\hat \C, 2}{\globaltriv}(\gamma_1, \gamma_2) = \lpot{\hat \C, 2}(\gamma_1, \gamma_2) + f_\globaltriv(\tau) + \const
    .
\end{equation}
The appearance of a function of the modulus, means that the existence of the infimum of $\lpotx{\hat \C, 2}{\globaltriv}(\gamma_1, \gamma_2)$ and the geometry of the minimizers depend on the choice of trivialization $\globaltriv$.

Recall that the variations in Theorem~\ref{prop:two_loop_variation} preserve the modulus $\tau$. The following result is a direct consequence of Equation~\eqref{eq:lpot_z_two}.
\begin{corollary}
    \label{cor:two_loop_variation_cft}
    Fix any trivialization $Z$ of the real determinant line bundle.
    Let $\nu$ be an infinitesimal Beltrami differential with compact support in the interior of $D_1 \cup D_2$ and let $f_j : \disk \to D_j$, $j = 1, 2$, be any Riemann mappings.
    Then,
    \begin{equation}
    \eval{\pdv{\varepsilon}}_{\varepsilon = 0}
        \lpotx{\hat \C, 2}{Z}(\omega^{\varepsilon \nu}(\gamma_1), \omega^{\varepsilon \nu}(\gamma_2))
        = - \frac{1}{3\pi} \left(
            \iint_{D_1} \nu \, \schwarzian{f_1^{-1}} \, \dz
            + \iint_{D_2} \nu \, \schwarzian{f_2^{-1}} \, \dz
        \right).
    \end{equation}
    The variation is independent of the choice of Riemann mappings $f_1$, $f_2$.
\end{corollary}
Hence we conclude that the infimum of generalized two-loop Loewner potentials $\lpotx{\hat \C, 2}{Z}$ is also obtained by minimizing among pairs of disjoint circles.
To find these, we follow a strategy similar to Proposition~\ref{prop:lpot_circles}.
To this end, we find the following criterion.
\begin{proposition}
    \label{prop:lpot_z_circles}
    Fix any trivialization $Z$ of the real determinant line bundle.
    There exists a pair of disjoint circles attaining the infimum of the Loewner potential $\lpotx{\hat \C, 2}{Z}$ if and only if 
    \begin{equation}
        \label{eq:lpot_z_circles}
        e^{-\frac{\pi}{3} \charge \tau}
        Z_{\flatmetric}(
            \annulus_\tau
        )
    \end{equation}
    has a global minimum in $\tau \in (0, \infty)$. Here, $\tau$ is the modulus of the annulus between the circles and $Z(\blank) = Z_g(\blank) [g]$ is the trivialization with respect to a metric $g$.
\end{proposition}
\begin{proof}
    By conformal invariance, without loss of generality, assume that $\gamma_1 = e^{-2\pi \tau} \, S^1$ and $\gamma_2 = S^1$.
    Fix a metric $g$ which is the flat metric $\flatmetric$ on $\disk$ and extends smoothly to $\hat \C$ (this choice is independent of $\tau$).
    Using the conformal anomaly $\conformalanomaly(-2\pi \tau, \flatmetric) = -\frac{\pi}{3} \tau$ from scaling $e^{-2\pi \tau} \disk$ to $\disk$, we find
    \begin{equation}
            \log Z_{\flatmetric}(e^{-2\pi \tau} \disk)
            =
            \log Z_{e^{-4 \pi \tau} \flatmetric}(\disk)
            =
            \log Z_{\flatmetric}(\disk)
            - \frac{\pi}{3} \charge \tau
            .
    \end{equation}
    By Equation~\eqref{eq:def_lpot_z_two} we have
    \begin{equation}
    \begin{aligned}
    \label{eq:lpot_z_flat}
        \frac{\charge}{2} \lpotx{\hat \C, 2}{Z}(\gamma_1, \gamma_2) &=
        \log \frac{
            Z_{\flatmetric}(e^{-2\pi \tau} \disk)
            Z_{\flatmetric}(
                \annulus_\tau
            )
            Z_g(\disk^*)
        }{
            Z_g(\hat \C)
        }
        \\ &=
            \log Z_{\flatmetric}(
                \annulus_\tau
            )
            - \frac{\pi}{3} \charge \tau
            + \const
    \end{aligned}
    \end{equation}
    Discarding the terms independent of $\tau$, we find that \eqref{eq:lpot_z_flat} being minimal at $\tau$ is equivalent to Equation~\eqref{eq:lpot_z_circles} being minimal.
\end{proof}

\begin{example}
    In Section~\ref{section:lpot_variation_modulus}, it was shown that the two-loop Loewner potential~\eqref{eq:lpot_def} associated to
    \begin{equation}
        Z_{\flatmetric}(\annulus_\tau) = \left(
            \frac{
                \detz{\restrict{\flatmetric}{\annulus_\tau}}
            }{
                B_{\flatmetric}(\annulus_\tau)
            }
        \right)^{-\charge/2}
    \end{equation}
    diverges to $-\infty$ as $\tau \to \infty$.
\end{example}

\begin{figure}
\centering
\begin{tikzpicture}
    \definecolor{color2}{HTML}{F1A340}
    \definecolor{color1}{HTML}{998EC3}

    \begin{axis}[
        xlabel={$q$},
        ylabel={$e^{\frac{\pi}{3} \charge \tau} \; \frac{q^{h - \frac{\charge}{24}}}{\phi(q)}$},
        xmin=0, xmax=1,
        ymin=0, ymax=10,
        xtick={0, 0.2, 0.4, 0.6, 0.8, 1},
        ytick={0, 5, 10},
        %legend pos=north west,
        %ymajorgrids=true,
        %grid style=dashed,
    ]
    
    \addplot[color=color1]
        coordinates {
        (0.010000, 10.10203)
        (0.020000, 7.218322)
        (0.030000, 5.957592)
        (0.040000, 5.217028)
        (0.050000, 4.719931)
        (0.060000, 4.359760)
        (0.070000, 4.085653)
        (0.080000, 3.869879)
        (0.090000, 3.695877)
        (0.100000, 3.553081)
        (0.110000, 3.434397)
        (0.120000, 3.334874)
        (0.130000, 3.250942)
        (0.140000, 3.179960)
        (0.150000, 3.119936)
        (0.160000, 3.069339)
        (0.170000, 3.026980)
        (0.180000, 2.991923)
        (0.190000, 2.963428)
        (0.200000, 2.940907)
        (0.210000, 2.923893)
        (0.220000, 2.912020)
        (0.230000, 2.905000)
        (0.240000, 2.902616)
        (0.250000, 2.904707)
        (0.260000, 2.911167)
        (0.270000, 2.921934)
        (0.280000, 2.936989)
        (0.290000, 2.956352)
        (0.300000, 2.980082)
        (0.310000, 3.008276)
        (0.320000, 3.041069)
        (0.330000, 3.078634)
        (0.340000, 3.121185)
        (0.350000, 3.168980)
        (0.360000, 3.222322)
        (0.370000, 3.281569)
        (0.380000, 3.347131)
        (0.390000, 3.419484)
        (0.400000, 3.499174)
        (0.410000, 3.586829)
        (0.420000, 3.683166)
        (0.430000, 3.789011)
        (0.440000, 3.905311)
        (0.450000, 4.033156)
        (0.460000, 4.173799)
        (0.470000, 4.328691)
        (0.480000, 4.499510)
        (0.490000, 4.688209)
        (0.500000, 4.897063)
        (0.510000, 5.128737)
        (0.520000, 5.386362)
        (0.530000, 5.673637)
        (0.540000, 5.994952)
        (0.550000, 6.355543)
        (0.560000, 6.761688)
        (0.570000, 7.220962)
        (0.580000, 7.742559)
        (0.590000, 8.337719)
        (0.600000, 9.020270)
        (0.610000, 9.807366)
        (0.620000, 10.720452)
        (0.630000, 11.786573)
        (0.640000, 13.040157)
        (0.650000, 14.525472)
        (0.660000, 16.300046)
        (0.670000, 18.439519)
        (0.680000, 21.044596)
        (0.690000, 24.251177)
        (0.700000, 28.245380)
        (0.710000, 33.286215)
        (0.720000, 39.740485)
        (0.730000, 48.137705)
        (0.740000, 59.258537)
        };

    \addplot[color=color2]
        coordinates {
        (0.000100, 102.703017)
        (0.010100, 5.950047)
        (0.020100, 3.650963)
        (0.030100, 2.691192)
        (0.040100, 2.142470)
        (0.050100, 1.780212)
        (0.060100, 1.520160)
        (0.070100, 1.322909)
        (0.080100, 1.167337)
        (0.090100, 1.041011)
        (0.100100, 0.936091)
        (0.110100, 0.847367)
        (0.120100, 0.771229)
        (0.130100, 0.705091)
        (0.140100, 0.647046)
        (0.150100, 0.595656)
        (0.160100, 0.549812)
        (0.170100, 0.508647)
        (0.180100, 0.471470)
        (0.190100, 0.437724)
        (0.200100, 0.406954)
        (0.210100, 0.378786)
        (0.220100, 0.352908)
        (0.230100, 0.329057)
        (0.240100, 0.307011)
        (0.250100, 0.286582)
        (0.260100, 0.267607)
        (0.270100, 0.249946)
        (0.280100, 0.233478)
        (0.290100, 0.218095)
        (0.300100, 0.203706)
        (0.310100, 0.190228)
        (0.320100, 0.177589)
        (0.330100, 0.165724)
        (0.340100, 0.154576)
        (0.350100, 0.144094)
        (0.360100, 0.134231)
        (0.370100, 0.124948)
        (0.380100, 0.116205)
        (0.390100, 0.107970)
        (0.400100, 0.100212)
        (0.410100, 0.092903)
        (0.420100, 0.086017)
        (0.430100, 0.079531)
        (0.440100, 0.073424)
        (0.450100, 0.067676)
        (0.460100, 0.062269)
        (0.470100, 0.057186)
        (0.480100, 0.052412)
        (0.490100, 0.047931)
        (0.500100, 0.043731)
        (0.510100, 0.039799)
        (0.520100, 0.036124)
        (0.530100, 0.032692)
        (0.540100, 0.029495)
        (0.550100, 0.026522)
        (0.560100, 0.023762)
        (0.570100, 0.021208)
        (0.580100, 0.018850)
        (0.590100, 0.016679)
        (0.600100, 0.014686)
        (0.610100, 0.012865)
        (0.620100, 0.011205)
        (0.630100, 0.009700)
        (0.640100, 0.008341)
        (0.650100, 0.007121)
        (0.660100, 0.006031)
        (0.670100, 0.005065)
        (0.680100, 0.004213)
        (0.690100, 0.003468)
        (0.700100, 0.002823)
        (0.710100, 0.002269)
        (0.720100, 0.001798)
        (0.730100, 0.001403)
        (0.740100, 0.001076)
        (0.750100, 0.000810)
        (0.760100, 0.000596)
        (0.770100, 0.000428)
        (0.780100, 0.000299)
        (0.790100, 0.000202)
        (0.800100, 0.000132)
        (0.810100, 0.000082)
        (0.820100, 0.000049)
        (0.830100, 0.000028)
        (0.840100, 0.000014)
        (0.850100, 0.000007)
        (0.860100, 0.000003)
        (0.870100, 0.000001)
        (0.880100, 0.000000)
        (0.890100, 0.000000)
        (0.900100, 0.000000)
        (0.910100, 0.000000)
        (0.920100, 0.000000)
        (0.930100, 0.000000)
        (0.940100, 0.000000)
        (0.950100, 0.000000)
        (0.960100, 0.000000)
        (0.970100, 0.000000)
        (0.980100, 0.000000)
        (0.990100, 0.000000)
        };

        \legend{{} {$\charge = 0$, $h = -1/2$}, {} {$\charge = 1$, $h = -1/2$}}
    \end{axis}
\end{tikzpicture}
\caption{
    Plots of two instances of Equation~\eqref{eq:lpot_z_circles}, where the partition function is the Verma module Virasoro character~\eqref{eq:virasoro_character}.
    Note that one of the functions has a global minimum for $q = e^{-\pi /\tau} \in (0, 1)$ while the other is minimal for $q \to 1$ which corresponds to $\tau \to \infty$, that is, the circles moving away from each other.
}
\label{fig:characters}
\end{figure}
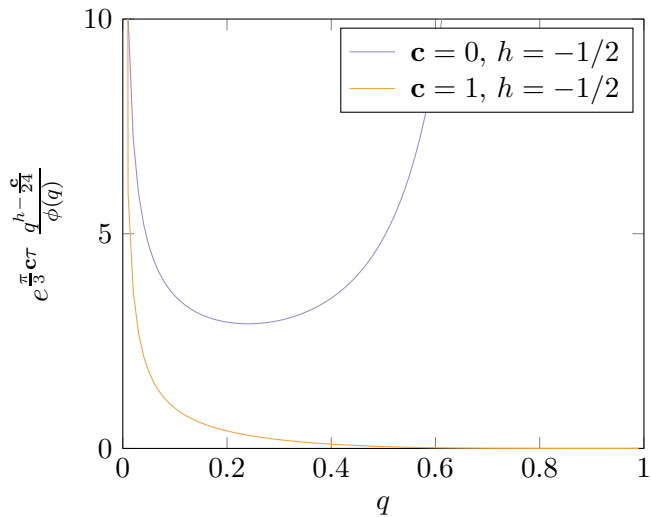

\begin{example}
\label{example:characters}
    By Example~\ref{example:trivializations}, Item~\ref{item:cft_trivialization}, $Z_{\flatmetric}(\annulus_\tau)$ may be given by CFT partition functions on annuli, which are studied in the work of Cardy on boundary conformal field theory (BCFT) \cite{cardy_fusion, cardy_boundary_conditions, cardy_bcft}.
    If the BCFT has a discrete spectrum $\cftspec \subset \R$, it is argued that the partition functions are of the general form\footnote{
        Cardy uses a periodic strip instead of an annulus. The conformal mapping between these is of the form $z \mapsto e^{\ii C z}$ where $C$ is a constant depending on $\tau$. These conformal maps give constant conformal factors. By the Gauss--Bonnet theorem, the conformal anomaly of such conformal transformations with respect to a flat metric vanishes (since the Euler characteristic of an annulus is $0$).
        Thus, the annulus partition functions agree.
    }
    \begin{equation}
        Z_{\flatmetric}(\annulus_\tau)
        =
        \sum_{h \in \cftspec} n_h \chi_{\charge, h}(q)
        , \qquad
        q = e^{-\frac{\pi}{\tau}}.
    \end{equation}
    The coefficients $n_h \in \Zpos$ may depend on the choice of two conformally invariant boundary conditions at the two boundary components of $\annulus_\tau$.
    The function $\chi_{\charge, h}(q)$ is the character of the irreducible highest-weight representation of the Virasoro algebra with highest weight~$h$,
    \begin{equation}
        \chi_{\charge, h}(q)
        = q^{\frac{\charge}{24}} \Tr q^{L_0}
        .
    \end{equation}

    For simplicity, we consider only the case of the character of a single Verma module of weight $h$ of the Virasoro algebra. In this case, the character is
    \begin{equation}
    \label{eq:virasoro_character}
        \chi_{\charge, h}(q) = \frac{q^{h - \frac{\charge}{24}}}{\phi(q)}
        =
        \frac{q^{h - \frac{\charge}{24}}}{\prod_{k = 1}^\infty (1 - q^k)}
        ,
    \end{equation}
    with the Euler function $\phi$ also appearing in Equation~\eqref{eq:euler}. Note that the representations may not be irreducible for certain values of $\charge$ and $h$ (which are called minimal models).
    Figure~\ref{fig:characters} compares plots of Equation~\eqref{eq:lpot_z_circles} for two values of $\charge$ and $h$ demonstrating the possibility of existence and non-existence of the minimizer of the two-loop Loewner potential.
    We leave the detailed study and interpretation of the existence of the minimum in the context of BCFT as a possible direction for future work.
\end{example}

\appendix

\section{Background}
\label{appendix:background}

\subsection{Brownian loop measure}
\label{appendix:blm}

The Brownian loop measure is an infinite measure on the set of continuous loops in $\hat \C$ \cite{loop_soup}.
On subdomains $D \subset \hat \C$, it is defined by restricting the Brownian loop measure in $\hat \C$, making it restriction invariant.
It is also conformally invariant, meaning that for conformally equivalent subdomains $D_1$ and $D_2$ of $\hat \C$, the pushforward of the measure on $D_1$ onto loops in $D_2$ is exactly the measure in $D_2$.
In this work, only the following quantity appears,
\begin{equation}
    \blmpair_D(V_1, V_2) = \text{total mass of Brownian loops in $D$ intersecting both $V_1$ and $V_2$},
\end{equation}
which is finite for subdomains $D$ with non-polar boundary and compact disjoint non-polar subsets $V_1, V_2 \subset D$.

For $D = \hat \C$ however, $\blmpair_D(V_1, V_2)$ is infinite due to the infinite mass of increasingly large loops. By removing a small disk (e.g.\ at infinity) and shrinking it to a point, $\blmpair_{\hat \C}(V_1, V_2)$ may be renormalized in the following way \cite{field_lawler},
\begin{equation}
\label{eq:blmpairr}
        \blmpairr(V_1, V_2)
        = \lim_{R \to \infty} \big(
            \blmpair_{R \, \disk}(V_1, V_2) - \log \log R
        \big) < \infty
        .
\end{equation}
The renormalization may be thought of as the removal of a small disk $R \disk^*$ at infinity, hence limiting the size of the loops.
In fact, the same limit is obtained by removing small disks at other points, implying that the limit remains invariant under Möbius transformations.

\subsection{One-loop SLE}
\label{appendix:one_loop_sle}

For $\kappa \in [0, 4)$, one-loop Schramm--Loewner evolution in a Riemann surface (with boundary) $\Sigma$ is a measure on the space of simple loops in $\Sigma$, which is the set of subsets $\gamma \subset \Sigma$ homeomorphic to $S^1$, equipped with the Hausdorff metric with respect to a conformal Riemannian metric on $\Sigma$.
See \cite{Zhan21} for a detailed construction, at least in the case of the Riemann sphere $\Sigma = \hat \C$ and for subdomains $\Sigma = D \subset \hat \C$.
In fact, the loop measure on subdomains $D \subset \hat \C$ is defined from the loop measure on $\hat \C$ by the \emph{restriction covariance property},
\begin{equation}
    \label{eq:restriction_covariance}
    \frac{
        \dd \sleloop{D}
    }{
        \dd \sleloop{\hat \C}
    }
    (\gamma)
    =
    \boldone_{\gamma \subset D}
    e^{\frac{\charge}{2} \blmpairr(\gamma, \hat \C \setminus D)}
    .
\end{equation}
The one-loop SLE measure on $\hat \C$ has recently been shown to be the unique nontrivial and locally finite measure on loops in $\hat \C$ satisfying \emph{conformal invariance} of the restricted measures $\sleloop{D}$ \cite{baverez_jego} (up to a scalar). Namely, local conformal invariance means that for any conformally equivalent subdomains $D_1$ and $D_2$ of $\hat \C$, the pushforward of $\sleloop{D_1}$ onto loops in $D_2$ is exactly $\sleloop{D_2}$.

The Loewner potential for the one-loop measure is
\begin{equation}
\label{eq:lpot_one_detz}
    \lpot{\hat \C}(\gamma) =
    \log \frac{
    \detz{\restrict{g}{\hat \C}}
    }{
    \detz{\restrict{g}{D_1}}
    \detz{\restrict{g}{D_2}}
    },
\end{equation}
where $D_1$ and $D_2$ are the connected components of $\hat \C \setminus \gamma$ and $g$ is any metric on $\hat \C$. The potential is independent of the metric by \eqref{eq:polyakov_alvarez} in Appendix~\ref{appendix:zeta}.

The relation to the  Loewner energy of a single loop is given by subtraction of the infimum of the potential,
\begin{equation}
    \lenergy{\hat \C}(\gamma) = 12\big(\lpot{\hat \C}(\gamma) - \inf_{\eta} \lpot{\hat \C}(\eta)\big).
\end{equation}
See \cite{peltola_wang} for more details on the distinction between Loewner potential and energy.
In \cite{carfagnini_wang}, it is shown that $\lenergy{\hat \C}$ is an Onsager--Machlup functional for $\sleloop{\hat \C}$. We reformulate this theorem in terms of the potential.
\begin{theorem}
\label{thm:onsager_machlup_one}
    For simple analytic loops $\gamma$ and $\xi$ in $\hat \C$, we have
    \begin{equation}
        \lim_{\varepsilon \to 0}
        \frac{\sleloop{\hat \C}(
            \loopnbhd{\varepsilon}{\gamma}
        )}{\sleloop{\hat \C}(
            \loopnbhd{\varepsilon}{\xi}
        )
        }
        =
        e^{
        \frac{\charge}{2} \left(
            \lpot{\hat \C}(\gamma)
            - \lpot{\hat \C}(\xi)
        \right)
        }
        ,
    \end{equation}
    where $\loopnbhd{\varepsilon}{\blank}$ is the $\varepsilon$-neighborhood of a loop as defined in Equation~\eqref{eq:loop_nbhd}.
\end{theorem}

\subsection{Zeta-regularized determinants of Laplacians}
\label{appendix:zeta}

In this section, we provide results on zeta-regularized determinants of Laplacians associated to compact Riemann surfaces with or without boundary $\Sigma$ and conformal metrics $g$ on $\Sigma$, which appear in this work.

Since $\Sigma$ is compact, the positive Laplacian or Laplace--Beltrami operator $\Delta_g$ with Dirichlet boundary condition has a discrete spectrum.
We first consider the case where $\partial \Sigma \neq \emptyset$.
In that case, the zeta-regularized determinant $\detz{g}$ is defined by analytic continuation of the associated spectral zeta function \cite{ray_singer}.
Most notably, we are interested in the change of $\detz{g}$ under conformal transformations replacing $g$ by $e^{2\sigma} g$ for $\sigma \in C^\infty(\Sigma, \R)$.
The change is given by the Polyakov--Alvarez anomaly formula \cite{polyakov, alvarez, OPS88},
\begin{equation}
    \label{eq:polyakov_alvarez}
    \begin{gathered}
    \frac{\detz{e^{2\sigma} g}}{\detz{g}}
    = e^{\paformula(\sigma, g)}, \\
    \paformula(\sigma, g)
    =
    -\frac{1}{6 \pi} 
    \iint_\Sigma \bigg(
    \frac{1}{2} |\nabla_g \sigma|_g^2 
    + R_g \sigma
    \bigg) \dvol{g}
    - \frac{1}{6 \pi}
    \int_{\partial \Sigma} \bigg(
        k_g \sigma + \frac{3}{2} N_g \sigma
    \bigg) \dboundary{g}
    ,
    \end{gathered}
\end{equation}
where $\nabla_g$, $R_g$, $k_g$ and $N_g$ are respectively the divergence, Gaussian curvature, boundary curvature and normal derivative with respect to the metric $g$.

A zeta-regularized determinant of the Laplacian may also be defined for conformal metrics on compact Riemann surfaces without boundary.
However, due to lack of a boundary condition, the associated Laplacian has a zero eigenvalue associated to the constant functions.
The definition of the zeta-regularized determinant excludes the zero eigenvalue and is denoted $\detzp{g}(\Sigma)$.
It also satisfies an anomaly formula, which comes with an additional term involving the volumes $\vol_g(\Sigma)$ and $\vol_{e^{2\sigma}g}(\Sigma)$ before and after the conformal transformation \cite{polyakov}.
To reconcile this, we define
\begin{equation}
    \label{eq:detz_boundary}
    \detz{g} \coloneqq \frac{\detzp{g}}{\vol_g(\Sigma)}.
\end{equation}
This definition turns Equation~\eqref{eq:polyakov_alvarez} into a unified Polyakov-Alvarez anomaly formula for surfaces with or without boundary.

In the case of a disk or an annulus and the flat metric $g = \flatmetric$, the explicit expressions for the zeta-regularized determinant of the Laplacian listed below were found in \cite{weisberger}.
The expressions involve the standard Riemannn zeta function $\zeta_R(s)$, defined by analytic continuation of $\zeta_R(s)=\sum_{k \geq 1} k^{-s}$, and the Euler function defined by
\begin{equation}
\label{eq:euler}
    \phi(x) = \prod_{k = 1}^\infty (1 - x^k)
    .
\end{equation}
For the flat metric on a disk $r \disk$ of radius $r > 0$, we have
\begin{equation}
\label{eq:det_disk}
    \log \detz{\restrict{\dz}{r \disk}}
    =
    - \frac{1}{6} \log 2
    - \frac{1}{2} \log \pi
    - \frac{1}{3} \log r
    - 2\zeta_R'(-1) - \frac{5}{12}
    .
\end{equation}
Also in the flat metric, on an annulus $A = \setsuchthat{z \in \C}{r_1 \leq |z| \leq r_2}$ for $r_2 > r_1 > 0$, we have
\begin{equation}
\label{eq:det_annulus}
    \log \detz{\restrict{\dz}{A}}
    =
    - \log \pi
    + \frac{1}{3} (\log r_1 - \log r_2)
    + \log (\log r_2 - \log r_1)
    + 2 \log \phi\left(\left(r_1/r_2\right)^2\right)
    .
\end{equation}

We use these explicit formulas to prove the following lemma, which describes a renormalization of the zeta-regularized determinants similar to that of the Brownian loop measure in Equation~\eqref{eq:blmpairr}. In fact, we relate the two in an application of the lemma in Proposition~\ref{prop:lpot_detz}.
Figure~\ref{fig:setup} illustrates the setup of the lemma.

\begin{lemma}
\label{lemma:det_lim}
Let $D_1$, $A$, $D_2$ be the domains separated by smooth loops $\gamma_1$ and $\gamma_2$ in $\hat \C$.
Assume that $0 \in D_1$ and $\infty \in D_2$.
For any metric $g$ on $\hat \C$, we have
    \begin{equation}
    \label{eq:det_lim}
    \begin{aligned}
        \lim_{R \to \infty} \Bigg(
            &\log \frac{
                \detz{\restrict{g}{D_1 \cup A}}
                \detz{\restrict{g}{(A \cup D_2) \setminus R \, \disk^*}}
            }{
                \detz{\restrict{g}{\hat \C \setminus R \, \disk^*}}
                \detz{\restrict{g}{A}}
            }
            - \log \log R
        \Bigg)
        \\
        = 
        &\log \frac{
            \detz{\restrict{g}{D_1 \cup A}}
            \detz{\restrict{g}{A \cup D_2}}
        }{
            \detz{\restrict{g}{\hat \C}}
            \detz{\restrict{g}{A}}
        }
        + \const,
    \end{aligned}
    \end{equation}
    where the constant is independent of $\gamma_1$, $\gamma_2$, and $g$.
\end{lemma}
\begin{proof}
    Let $\Psi : D_2 \cup A \to \disk^*$ be the Riemann mapping such that $\Psi(\infty) = \infty$ and $\Psi'(\infty) > 0$ and define $B_R = \Psi^{-1}(R \, \disk^*)$ for $R > 1$ such that $R \, \disk^* \subset D_2$.

    Note that both the left- and right-hand sides of Equation~\eqref{eq:det_lim} are independent of the smooth Riemannian metric $g$ on $\hat \C$ before taking the limit, that is, for fixed $R > 0$.
    Let $\Psi$ extend to a diffeomorphism $\Psi : \hat \C \to \hat \C$. This is possible since the boundary $\partial (D_2 \cup A) = \gamma_1$ is smooth and thus $\Psi$ extends to a smooth homeomorphism on $\gamma_1$.
    Define $g = \Psi_* h$, where $h$ is the flat metric $\flatmetric$ on a neighborhood of
    $R \disk$
    and with smooth continuation on $R \disk^*$ (for instance, by interpolating with the round metric near $\infty$).
    Then --- still before taking the limit --- the difference of the left- and right-hand sides of Equation~\eqref{eq:det_lim} is
    \begin{equation}
    \label{eq:det_before_lim_difference}
    \begin{aligned}
        &
        \log \frac{
            \detz{\restrict{g}{\hat \C}}
            \detz{\restrict{g}{(A \cup D_2) \setminus B_R}}
        }{
            \detz{\restrict{g}{\hat \C \setminus B_R}}
            \detz{\restrict{g}{A \cup D_2}}
        }
        - \log \log R
        \\ = \; &
        \log \frac{
            \detz{\restrict{\Psi^* g}{\hat \C}}
            \detz{\restrict{\Psi^* g}{\{1 \leq |z| \leq R\}}}
        }{
            \detz{\restrict{\Psi^* g}{R \, \disk}}
            \detz{\restrict{\Psi^* g}{\disk^*}}
        }
        - \log \log R
        \\ = \; &
        \log \frac{
            \detz{\restrict{h}{\hat \C}}
            \detz{\restrict{\flatmetric}{\{1 \leq |z| \leq R\}}}
        }{
            \detz{\restrict{\flatmetric}{R \, \disk}}
            \detz{\restrict{h}{\disk^*}}
        }
        - \log \log R
        \\ = \; &
        \frac{1}{6} \log 2
        + \frac{1}{2} \log \pi
        + \frac{1}{3} \log R
        + 2\zeta_R'(-1) + \frac{5}{12}
        \\ &
        + \log \detz{\restrict{h}{\hat \C}}
        - \log\detz{\restrict{h}{\disk^*}}
        \\ &
        - \log \pi
        - \frac{1}{3} \log R
        + \log \log R
        + 2 \log \phi(R^{-2})
        \\ &
        - \log \log R
        \\ = \; &
        \frac{1}{6} \log 2
        + \frac{1}{2} \log \pi
        + 2\zeta_R'(-1) - \frac{5}{12}
        - \log \pi
        \\ &
        + \log \detz{\restrict{h}{\hat \C}}
        - \log\detz{\restrict{h}{\disk^*}}
        \\ &
        + 2 \log \phi(R^{-2})
        \\ \xrightarrow{r \to \infty} \; &
        \frac{1}{6} \log 2
        + \frac{1}{2} \log \pi
        + 2\zeta_R'(-1) - \frac{5}{12}
        - \log \pi
        \\ &
        + \log \detz{\restrict{h}{\hat \C}}
        - \log\detz{\restrict{h}{\disk^*}}
    \end{aligned}
    \end{equation}
    Since the last two terms are invariant under conformal changes of $h$ on $\disk^*$, we conclude that the limit is a constant, which is indeed independent of $\gamma_1, \gamma_2$ and $g$.
\end{proof}

\section{The real determinant line bundle and Loewner potentials}
\label{appendix:detrc}

In this appendix, we briefly introduce the real determinant line bundle and highlight the relation to Loewner energy as found in \cite[Theorem~3.8]{maibach_peltola}.
The generalization of this relation leads to a notion of Loewner potential for configurations of loops on Riemann surfaces.
To show that it is compatible with the definition of the two-loop Loewner potential on the Riemann sphere in Section~\ref{section:two_loop_sle}, we examine a particular choice of trivialization of the real determinant line bundle based on the zeta-regularized determinant of the Laplacian.
Other trivializations can be obtained, for example, from CFT partition functions.
These yield additional definitions of the Loewner potential, which are discussed further in Section~\ref{section:cft_partition_functions}.

The real determinant line of a compact Riemann surface $\Sigma$ (with boundary) is a real half-line of equivalence classes
\begin{equation}
    \Detrpc(\Sigma) = \setsuchthat{\lambda [g]}{\text{$\lambda \in \Rp$, $g$ conformal metric on $\Sigma$}},
\end{equation}
under the relation
\begin{gather}
\label{eq:detrc_relation}
\lambda [e^{2\sigma} g] = \lambda e^{\charge \conformalanomaly(\sigma, g)} [g], \\
\label{eq:conformal_anomaly}
\conformalanomaly(\sigma, g) =
\frac{1}{12 \pi} \iint_\Sigma \bigg(
\frac{1}{2} |\nabla_g \sigma|_g^2 + R_g \sigma
\bigg) \dvol{g}
+ \frac{1}{12 \pi} \int_{\partial_\Sigma} k_g \sigma \, \dboundary{g}.
\end{gather}
The constant $\charge \in \R$ is called the \emph{central charge} of the real determinant line and $\conformalanomaly(\sigma, g)$ is the \emph{conformal anomaly} of the conformal change of the metric $g$ to $e^{2\sigma} g$ by a function $\sigma \in C^\infty(\Sigma, \R)$.

If $f : \Sigma_1 \to \Sigma_2$ is an isomorphism of Riemann surfaces, then the real determinant lines are isomorphic by pushing forward the metric, $\lambda [g] \mapsto \lambda e^{\charge \conformalanomaly(\log|f'|)} [f_* g]$. The factor makes the pushforward invariant under automorphisms.
Therefore, the real determinant lines form $\Rp$-bundles
$\Detrpc(\moduli{\genus}{\boundaries})$
over the moduli spaces $\moduli{\genus}{\boundaries}$ of compact Riemann surfaces of genus $\genus$ with $\boundaries$ boundary components.
The moduli spaces consist of equivalence classes $[\Sigma] \in \moduli{\genus}{\boundaries}$ of Riemann surfaces of the respective type up to isomorphism.

The smooth structure of the real determinant line bundles is determined by the choice of a (globally defined) trivialization $\globaltriv$.
Since the choice of $\globaltriv$ is essential to the following discussion of generalized Loewner potentials, we 
emphasize the definition and give a few examples below.
\begin{definition}
    A \emph{trivialization} of the real determinant line bundle is a section
    \begin{equation}
        \globaltriv : \moduli{\genus}{\boundaries} \to \Detrpc(\moduli{\genus}{\boundaries}).
    \end{equation}
    Given any Riemann surface $\Sigma$ such that $[\Sigma] \in \moduli{\genus}{\boundaries}$ and any conformal metric $g$ on $\Sigma$, there exists $\globaltriv_g(\Sigma) > 0$ such that 
    \begin{equation}
        \globaltriv([\Sigma]) = \globaltriv_g(\Sigma) [g] \in \Detrpc(\Sigma).
    \end{equation}
\end{definition}
The relevant trivializations for this work are the following.
\begin{example}
\leavevmode
\makeatletter
\@nobreaktrue
\makeatother
\label{example:trivializations}
\begin{enumerate}
    \item
    \label{item:zeta_trivialization}
    \textnormal{\textbf{Zeta-regularized determinants:}}
    The Polyakov--Alvarez anomaly formula~\eqref{eq:polyakov_alvarez} and the conformal anomaly~\eqref{eq:conformal_anomaly} are related by a boundary term
    \begin{equation}
        \conformalanomaly(\sigma, g) + \frac{1}{2} \paformula(\sigma, g) = - \frac{1}{8\pi} \int_{\partial \Sigma} N_g \sigma \dboundary{g}.
    \end{equation}
    To obtain a trivialization of $\Detrc$, we introduce the functional
    \begin{equation}
        \boundaryterm{g}(\Sigma) = e^{\frac{1}{4\pi} \int_{\partial_\Sigma} k_g \dboundary{g}}
        , \qquad
        \frac{
            \boundaryterm{e^{2\sigma} g}(\Sigma)
        }{
            \boundaryterm{g}(\Sigma)
        }
        =
        e^{\frac{1}{4\pi} \int_{\partial \Sigma} N_g \sigma \dboundary{g}},
    \end{equation}
    with precisely this covariance.
    Therefore, the following elements of $\Detrpc$ are independent of the choice of conformal metric on~$\Sigma$,
    \begin{equation}
    \label{eq:trivzeta}
    \globaltrivzeta(\Sigma) = \left(
        \frac{\detz{g}}{\boundaryterm{g}(\Sigma)}
    \right)^{-\charge/2}
    [g] \in \Detrpc(\Sigma).
    \end{equation}
    As explained in Section~\ref{section:trivialization_sle}, this is the trivialization that is often implicitly used in the probabilistic study of SLE.
    \item
    \label{item:cft_trivialization}
    \textnormal{\textbf{CFT partition functions:}}
    Given a (boundary) conformal field theory of central charge $\charge$, which comes with partition functions $Z_g(\Sigma)$, define
    \begin{equation}
        \globaltrivcft(\Sigma) = Z_g(\Sigma) [g] \in \Detrpc(\Sigma).
    \end{equation}
    This element is independent of the choice of conformal metric $g$ by the Weyl covariance property of the partition functions \cite{gawedzki_lectures, lcft_review}.
    \item
    \textnormal{\textbf{Constant curvature metrics:}}
    For example, let $g_0(\Sigma)$ be the unique metric with constant curvature $+1$, $-1$, or $0$, such that the geodesic curvature of the boundary component vanishes \cite{OPS88}. Then, a trivialization is defined by
    \begin{equation}
        Z(\Sigma) = [g_0(\Sigma)].
    \end{equation}
\end{enumerate}
\end{example}

The construction that leads to generalized Loewner potentials involves the sewing isomorphisms $\sewisoloops{\vec \gamma}$ on real determinant lines inspired by the work of Segal~\cite{segal}.
These are multilinear maps on the determinant lines associated to the connected components of $\Sigma$ minus non-intersecting smooth loops $\vec \gamma = (\gamma_1, \dots, \gamma_N)$ in $\Sigma$,
\begin{equation}
\label{eq:sewing_iso}
\begin{aligned}
    \sewisoloops{\Sigma, \vec \gamma} : \bigotimes_{A \in \pi_0(\Sigma \setminus \vec \gamma)} \Detrpc(A) &\to \Detrpc(\Sigma) \\
    \bigotimes_{A \in \pi_0(\Sigma \setminus \vec \gamma)} \lambda_A [\restrict{g}{A}] &\mapsto \Big( \prod_{A \in \pi_0(\Sigma \setminus \vec \gamma)}\lambda_A \Big) [g].
\end{aligned}
\end{equation}
This sewing isomorphism works by picking a conformal metric $g$ on $\Sigma$ and expressing elements of the determinant lines of the connected components $A$ relative to the restricted metrics $\restrict{g}{A}$. By the locality of the conformal anomaly \eqref{eq:conformal_anomaly}, this definition is independent of the choice of $g$, see also \cite[Section~3.1]{maibach_peltola}.

Furthermore, the real determinant line of $N$ non-intersecting smooth loops $\vec \gamma$ in $\Sigma$ is defined as
\begin{equation}
    \Detrpc(\Sigma, \vec \gamma) =
    \Detrpc(\Sigma)
    \otimes
    \bigotimes_{A \in \pi_0(\Sigma \setminus \vec \gamma)}
    \big(
    \Detrpc(A)
    \big)^\vee
    .
\end{equation}
Applying the sewing isomorphisms \eqref{eq:sewing_iso} to the determinant line of $\vec \gamma$ as $\id \otimes (\sewisoloops{\Sigma, \gamma})^\vee$, we obtain elements $\Detrpc(\Sigma) \otimes \Detrpc(\Sigma)^\vee$.
The evaluation of the first component in the dual, denoted by $\evaluation$, yields a positive real number.
We denote the composition of these operations by
\begin{equation}
\begin{aligned}
    \evaluation_{\Sigma, \vec \gamma} = \evaluation \circ \sewisoloops{\vec \gamma} : \Detrpc(\vec \gamma, \Sigma) &\to \Rp
    .
\end{aligned}
\end{equation}
It is the logarithm of these numbers that define the generalized Loewner potential.
\begin{definition}
\label{def:loewner_potential_Z}
    The \emph{Loewner potential} of non-intersecting loops $\vec \gamma = (\gamma_1, \dots, \gamma_N)$ in $\Sigma$ with respect to a trivialization $\globaltriv$ of $\Detrpc$ is
    \begin{equation}
        \frac{\charge}{2} \lpotx{\Sigma}{\globaltriv}(\vec \gamma) = 
        \log \evaluation_{\Sigma, \vec \gamma}\Big(
            \;
            \globaltriv(\Sigma)
            \otimes
            \bigotimes_{A \in \pi_0(\Sigma \setminus \vec \gamma)}
            \globaltriv(A)^\vee
            \;
        \Big).
    \end{equation}
    Concretely, for $\charge \neq 0$ and $\globaltriv(\Sigma) = \globaltriv_g(\Sigma) [g]$ we have
    \begin{equation}
        \frac{\charge}{2}
        \lpotx{\Sigma}{\globaltriv}(\vec \gamma)
        =
        \log \globaltriv_g(\Sigma)
        - \sum_{A \in \pi_0(\Sigma \setminus \vec \gamma)} \log \globaltriv_g(A)
        .
    \end{equation}

\end{definition}
In particular, for $\charge \neq 0$, the Loewner potential of a single loop $\gamma$ separating $\hat \C$ into $D_1$ and $D_2$ is
\begin{equation}
    \label{eq:def_lpot_z_one}
    \lpotx{\hat \C}{Z}(\gamma) =
    \frac{2}{\charge} 
    \log \frac{
        Z_g(D_1)
        Z_g(D_2)
    }{
        Z_g(\hat \C)
    }
    .
\end{equation}
For two loops $\gamma_1$ and $\gamma_2$ bounding disks $D_1$ and $D_2$ and an annulus $A$, we recover Equation~\eqref{eq:def_lpot_z_two}.

\DeclareEmphSequence{}
\printbibliography

\end{document}